\documentclass[article]{amsart}
\usepackage[english]{babel}
  \usepackage[
    backend=biber,
    style=alphabetic,
    doi=false,
    isbn=false,
    url=false,
    eprint=false
  ]{biblatex}
\usepackage{amsfonts}
\usepackage{microtype}
\usepackage[T1]{fontenc}
\usepackage[all,arc]{xy}
\usepackage{enumerate}
\usepackage{fancyhdr}
\usepackage[utf8]{inputenc}
\usepackage{amsfonts,empheq}
\usepackage{extpfeil}
\usepackage{fullpage}
\usepackage{amssymb}
\usepackage{comment}
\usepackage{hyperref} 
\usepackage{setspace}
\hypersetup{%
  colorlinks=true,%
  linkcolor=black,%
  citecolor=black,%
  filecolor=black,%
  menucolor=black,
  urlcolor=black,%
  pdfnewwindow=true,%
  pdfstartview=FitBH}   
\usepackage{listings}
\usepackage{url}
\usepackage{fancyhdr}
\usepackage{graphicx}
\usepackage{geometry}
\usepackage{amsmath}
\usepackage{enumitem}
\usepackage{mathtools}
\usepackage{color}
\usepackage{mathrsfs}
\usepackage{tikz-cd}
\usepackage{graphicx}
\usepackage{etoolbox}
\usepackage{amsthm}
\usepackage{lmodern}
\allowdisplaybreaks

\usepackage{titlesec}
\makeatletter
\renewcommand\section{\@startsection{section}{3}{\z@}%
                                     {-3.25ex\@plus -1ex \@minus -.2ex}%
                                     {1.5ex \@plus .2ex}
                                     {\normalfont\LARGE\bfseries}}
\renewcommand\subsection{\@startsection{subsection}{3}{\z@}%
                                     {-3.25ex\@plus -1ex \@minus -.2ex}%
                                     {-1.5ex \@plus -.2ex}
                                     {\normalfont\normalsize\bfseries}}
\patchcmd{\@maketitle}
  {\ifx\@empty\@dedicatory}
  {\ifx\@empty\@date \else {\vskip3ex \centering\footnotesize\@date\par\vskip1ex}\fi
   \ifx\@empty\@dedicatory}
  {}{}
\patchcmd{\@adminfootnotes}
  {\ifx\@empty\@date\else \@footnotetext{\@setdate}\fi}
  {}{}{}
\makeatother

\newtheorem{theorem}{Theorem}[subsection]
\newtheorem{corollary}[theorem]{Corollary}
\newtheorem{lemma}[theorem]{Lemma}

\newtheorem{proposition}[theorem]{Proposition}

\newtheorem{remark}[theorem]{Remark}

\theoremstyle{definition}

\newtheorem{definition}{Definition}[section]

\newcommand{\sub}{\subseteq}

\newcommand{\sm}{\setminus}
\newcommand{\CC}{\mathbb{C}}
\newcommand{\OO}{\mathbb{O}}

\newcommand{\ZZ}{\mathbb{Z}}

\newcommand{\PP}{\mathbb{P}}

\newcommand{\HH}{\mathbb{H}}

\newcommand{\kk}{\mathcal{K}}
\newcommand{\xx}{\mathcal{X}}
\newcommand{\pp}{\mathcal{P}}
\newcommand{\bb}{\mathcal{B}}
\newcommand{\dd}{\mathcal{D}}

\newcommand{\cc}{\mathcal{C}}
\newcommand{\mm}{\mathcal{M}}

\newcommand{\uu}{\mathcal{U}}
\newcommand{\oo}{\mathcal{O}}
\newcommand{\aaa}{\mathcal{A}}

\newcommand{\Jj}{\mathfrak{J}}

\newcommand{\Gg}{\mathfrak{g}}
\newcommand{\gl}{\mathfrak{gl}}
\newcommand{\sll}{\mathfrak{sl}}
\newcommand{\spp}{\mathfrak{sp}}
\newcommand{\Hh}{\mathfrak{h}}
\newcommand{\ol}{\overline}
\newcommand{\ra}{\rightarrow}
\newcommand{\lra}{\longrightarrow}
\newcommand{\hra}{\hookrightarrow{}}

\newcommand\blfootnote[1]{%
  \begingroup
  \renewcommand\thefootnote{}\footnote{#1}%
  \addtocounter{footnote}{-1}%
  \endgroup
}

\DeclareMathOperator{\ad}{ad}
\DeclareMathOperator{\im}{Im}
\DeclareMathOperator{\Lie}{Lie}

\DeclareMathOperator{\Ch}{Ch}
\DeclareMathOperator{\Tr}{Tr}

\DeclareMathOperator{\End}{End}
\DeclareMathOperator{\Der}{Der}

\DeclareMathOperator{\Sym}{Sym}
\DeclareMathOperator{\Hom}{Hom}
\DeclareMathOperator{\Spec}{Spec}

\DeclareMathOperator{\Id}{Id}

\newcommand{\Ann}{\mathrm{Ann}}

\newcommand{\gr}{\mathrm{gr}}
\newcommand{\nn}{\mathcal{N}}

\title{Almost commuting scheme of symplectic matrices and quantum Hamiltonian reduction}
\author{Pallav Goyal}

\begin{document}



\maketitle

\begin{abstract}
Losev introduced the scheme $X$ of almost commuting elements (i.e., elements commuting upto a rank one element) of $\mathfrak{g}=\mathfrak{sp}(V)$ for a symplectic vector space $V$ and discussed its algebro-geometric properties. We construct a Lagrangian subscheme $X^{nil}$ of $X$ and show that it is a complete intersection of dimension $\text{dim}(\mathfrak{g})+\frac{1}{2}\text{dim}(V)$ and compute its irreducible components.

We also study the quantum Hamiltonian reduction of the algebra $\mathcal{D}(\mathfrak{g})$ of differential operators on the Lie algebra $\mathfrak{g}$ tensored with the Weyl algebra with respect to the action of the symplectic group, and show that it is isomorphic to the spherical subalgebra of a certain rational Cherednik algebra of Type $C$.  We contruct a category $\cc_c$ of $\dd$-modules whose characteristic variety is contained in $X^{nil}$ and construct an exact functor from this category to the category $\oo$ of the above rational Cherednik algebra. Simple objects of the category $\cc_c$ are mirabolic analogs of Lusztig's character sheaves.
\end{abstract}

\section{Introduction}

\blfootnote{

  P. Goyal: Department of Mathematics, University of Chicago,
    Chicago, IL, 60637\par\nopagebreak
  \textit{E-mail address}: \texttt{pallav@math.uchicago.edu, pallavg@ucr.edu} \par\nopagebreak
    ORCID: 0000-0002-8983-0523
}
\subsection{}
Let $V:=\CC^{2n}$ be a symplectic vector space and let $\Gg$ denote the Lie algebra $\mathfrak{sp}(V)=\mathfrak{sp}_{2n}$. The almost commuting scheme $X$ of $\Gg$ was defined by Losev in \cite{L} as the closed subscheme of $\Gg\times\Gg\times V$ defined by the ideal $I$ generated by the matrix entries of $[x,y]+i^2$, i.e., by all functions of the form $(x,y,i)\mapsto \lambda([x,y]+i^2)$ for $\lambda\in \Gg^*$. Here, we use the fact that $\Sym^2(V)$ can be identified with $\mathfrak{sp}(V)$ to view $i^2$ as an element of $\mathfrak{sp}(V)$. The geometrical properties of $X$ were studied by Losev who showed that:

\begin{theorem} [\cite{L}] \label{theo:Losmain}
The scheme $X$ is reduced, irreducible and a complete intersection of dimension $2n^2+3n=\dim(\Gg)+\dim(V)$.
\end{theorem}

In this paper, we consider the reduced subscheme $X^{nil}$ of $X$ defined as:
\[X^{nil}:=\{(x,y,i)\in \Gg\times \Gg\times V: [x,y]+i^2=0 \text{ and } y \text{ is nilpotent}\}.\]
This definition is motivated by the notion of character sheaves, first defined by Lusztig (see \cite{Lus, Lus1, Lus2, Lus3, Lus4}). It was shown by Mirkovi\'c and Vilonen \cite{MV} and Ginzburg \cite{Gi} that over $\CC$, a character sheaf on a reductive algebraic group $K$ can be defined as an $Ad(K)$-equivariant perverse sheaf $M$ on $K$, such that the corresponding characteristic variety lies in the nilpotent locus $K\times\nn \sub K\times\mathfrak{k}^*$, where $\mathfrak{k}=\Lie(K)$ and $\nn\sub\mathfrak{k}^*\simeq\mathfrak{k}$ is the nilpotent cone. Constructions analogous to $X^{nil}$ were done in \cite{GG, FG, FG2} to provide `mirabolic' analogs of these character sheaves in Type $A$.

We fix some notation. It is known (for example, see \cite[Theorem 5.1.3]{DC}) that nilpotent conjugacy classes in $\Gg$ are parametrized by the partitions $\lambda$ of $2n$ in which every odd part appears an even number of times. Let $P_n$ be the set of all such partitions and let $\pp_n$ denote the subset of those partitions in $P_n$ in which all the parts are even. For each $\lambda\in P_n$, let $\nn_{\lambda}$ denote the corresponding nilpotent conjugacy class in $\Gg$. Define for each $\lambda\in P_n$:
\[X_{\lambda}:=\{(x,y,i)\in X^{nil}: y\in \nn_{\lambda}\}.\]
Let $\ol{X_{\lambda}}$ denote the closure of $X_{\lambda}$ in $X^{nil}$.

Note that we can identify $\Gg\times \Gg\times V$ with $T^*(\Gg)\times V$ using the trace form on $\Gg$. This gives $\Gg\times \Gg\times V$ a natural symplectic structure. Our first main result reads:
\begin{theorem} \label{theo:subsch}
\begin{enumerate} [label=(\alph*)]
\item The scheme $X^{nil}$ is a complete intersection in $\Gg\times\Gg\times V$ of dimension $2n^2+2n$. The irreducible components of $X^{nil}$ are exactly given by the $\ol{X_{\lambda}}$ for $\lambda\in\pp_n$.
\item With the natural symplectic structure, $X^{nil}$ is a Lagrangian subscheme of $\Gg\times\Gg\times V$.
\end{enumerate}
\end{theorem}

A similar Lagrangian subscheme was constructed in \cite{GG} in the context of the almost commuting scheme of the Lie algebra $\gl_n$. Using Theorem~\ref{theo:subsch}, we provide an independent proof of Theorem~\ref{theo:Losmain} in the style of \cite{GG}, that eliminates the use of results from \cite{Los2}. Furthermore, the scheme $X^{nil}$ will be used in $\mathsection\ref{sec:QHRF}$ to provide a mirabolic analog of Lusztig's character sheaves in Type $C$.

\subsection{}
In the second half of the paper, we discuss some Hamiltonian reduction problems arising in the context of the scheme $X$ and some other related schemes. For this, we define the following subschemes of $\Gg\times \Gg=\Spec(\CC[\Gg\times \Gg])=\Spec(\CC[x,y])$. Consider the commuting scheme $C$ which is the (not necessarily reduced) subscheme of $\Gg\times \Gg$ defined by the ideal of $\CC[x,y]$ generated by the matrix entries of the commutator $[x,y]$. Next, define the scheme $A$ to be the (not necessarily reduced) subscheme of $\Gg\times \Gg$ defined by the ideal of $\CC[x,y]$ generated by all the $2\times 2$ minors of the commutator $[x,y]$.

Note that the set of $\CC$-points of the underlying reduced subscheme of $C$ consists of pairs of elements of $\Gg$ that commute with each other, whereas that of $A$ consists of pairs of elements of $\Gg$ whose commutator has rank less than or equal to one. The commuting scheme $C$ is of wide interest, and its geometrical properties (most notably, its reducedness) are largely unknown. It is known that $C$ is irreducible (see \cite{Ri}).

The schemes $X, C$ and $A$ have an action of the group $G=Sp(V)$ obtained by the adjoint action on $\Gg$ and the natural action on $V$. Hence, we can consider the respective categorical quotients of these schemes by the action of $G$. While it isn't known if $C$ is reduced, it was shown in \cite{L} that there is an isomorphism:
\[C /\!/ G\lra X/\!/G,\]
which, paired with Theorem~\ref{theo:Losmain}, implies that $C/\!/G$ is reduced. (That $C/\!/G$ is reduced was deduced independently in \cite{NC} slightly earlier, by proving a version of the Chevalley restriction theorem for the commuting scheme of $\Gg$.) We extend this isomorphism to show that $C/\!/G \simeq X/\!/G \simeq A/\!/G$. In fact, we prove the following stronger result:
\begin{theorem} \label{theo:red}
We have an isomorphism of schemes:
\[X/\!/\{\pm 1\}\lra A/\!/\{\pm 1\} = A,\]
where $\{\pm 1\}\sub G$ is the center of the symplectic group. In particular, the scheme $A$ is reduced.
\end{theorem}
An analog of the isomorphism $X/\!/G \simeq A/\!/G$ for the Lie algebra $\gl_n$ was proved in \cite{GG}. The theorem is deduced from a linear algebraic lemma (Lemma~\ref{lem:maintech}), which also implies an algebro-geometric analog of the `shifting trick' in the theory of Hamiltonian reduction, that is well-known in the differential-geometric setting (see \cite{GS}, \cite{CS}).

\subsection{}
The above categorical quotients can (and will) all be viewed as classical Hamiltonian reductions of certain schemes under the action of the group $G$:
\begin{itemize}
\item The scheme $X/\!/G$ is the reduction of the scheme $\Gg\times \Gg\times V$ with respect to $G$ at $0$.
\item The scheme $C/\!/G$ is the reduction of the scheme $\Gg\times \Gg$ with respect to $G$ at $0$.
\item The scheme $A/\!/G$ is the reduction of the scheme $\Gg\times \Gg$ with respect to $G$ at the closure of the orbit of rank 1 matrices in $\Gg\simeq \Gg^*$.
\end{itemize}

So, we can try to study the non-commutative or quantum analogs of these classical reductions. For this, let $\uu\Gg$ denote the universal enveloping algebra of $\Gg$, let $\dd(\Gg)$ denote the algebra of polynomial differential operators on $\Gg$ and let $W_{2n}$ denote the Weyl algebra on $2n$ variables, which is the algebra of polynomial differential operators on the affine $n$-space. Then, $\dd(\Gg)$ is a quantization of $\CC[\Gg\times\Gg]\simeq \CC[\Gg\times\Gg^*]$, whereas $W_{2n}$ is a quantization of $\CC[V]$. Both the algebras $\dd(\Gg)$ and $W_{2n}$ have a natural $\Gg$-action (and, thus, so does their tensor product.) So, we get quantum co-moment maps corresponding to this $\Gg$-action:
\[\Theta_0: \uu\Gg \lra \dd(\Gg),\]
\[\Theta_2: \uu\Gg \lra \dd(\Gg)\otimes W_{2n}.\]
(These maps are elaborated upon in $\mathsection{\ref{sec:quant}}$.) Then, we can consider the following non-commutative algebras:
\begin{itemize}
\item The reduction $\Big((\dd(\Gg)\otimes W_{2n})/(\dd(\Gg)\otimes W_{2n})\cdot\Theta_2(\Gg)\Big)^{\Gg}$ of $\dd(\Gg)\otimes W_{2n}$ at the augmentation ideal of $\uu\Gg$.
\item The reduction $\Big(\dd(\Gg)/\dd(\Gg)\cdot\Theta_0(\Gg)\Big)^{\Gg}$ of $\dd(\Gg)$ at the augmentation ideal of $\uu\Gg$.
\item The reduction $\Big(\dd(\Gg)/\dd(\Gg)\cdot\Theta_0(\kk)\Big)^{\Gg}$ of $\dd(\Gg)$ at the unique primitive ideal $\kk\sub\uu\Gg$ such that $\gr(\kk)\sub\CC[\Gg^*]$ is the defining ideal of the orbit of rank 1 matrices.
\end{itemize}
(Each of these algebras is discussed in detail in $\mathsection\ref{sec:quant}.$)

The algebra $\Big(\dd(\Gg)/\dd(\Gg)\cdot\Theta_0(\Gg)\Big)^{\Gg}$ has been studied classically by Harish-Chandra (see \cite{HC}) who constructed a surjective algebra homorphism called the `radial parts' homomorphism:
\[\dd(\Gg)^{\Gg} \lra \dd(\Hh)^{W},\]
where $\Hh$ is a Cartan subalgebra of $\Gg$ and $W$ is the Weyl group. The kernel of this homomorphism was shown to be precisely $(\dd(\Gg)\cdot\Theta_0(\Gg))^{\Gg}$ in the works of Wallach \cite{Wa} and Levasseur and Stafford \cite{LS1, LS2}, implying that the algebra $\Big(\dd(\Gg)/\dd(\Gg)\cdot\Theta_0(\Gg)\Big)^{\Gg}$ is isomorphic to $\dd(\Hh)^W$.

In this paper, we'll discuss the other two quantum Hamiltonian reduction problems. For this, we recall the rational Cherednik algebra $H_c$ of Type $C$, first defined in \cite{EG}. Here, the parameter $c=(c_{long},c_{short})$ lies in $\CC^2$. Let $e=\frac{1}{|W|}\sum_{w\in W}w$ be the averaging idempotent of the Weyl group $W$ and consider the spherical subalgebra $eH_ce\sub H_c$ of the Cherednik algebra. (The notation is elaborated on in $\mathsection{\ref{sec:RCA}}$.)

We prove the following theorem about these algebras:
\begin{theorem} \label{theo:quant}
We have algebra isomorphisms:
\[\Big((\dd(\Gg)\otimes W_{2n})/(\dd(\Gg)\otimes W_{2n})\cdot\Theta_2(\Gg)\Big)^{\Gg} \simeq \Big(\dd(\Gg)/\dd(\Gg)\cdot \Theta_0(\kk)\Big)^{\Gg} \simeq eH_ce,\]
for the parameter $c=(-1/4,-1/2)$.
\end{theorem}

Analogs of this theorem were proved in \cite{EG} and \cite{GG} in the $\gl_n$-setting. This theorem shows that for the very special choice of the parameter $c=(-1/4,-1/2)$, the spherical subalgebra $eH_ce$ can be obtained as a quantum Hamiltonian reduction of the ring of differential operators on $\Gg$. The proof of the first isomorphism in the theorem follows from Lemma~\ref{lem:maintech}. The proof of the second isomorphism employs a generalization of the radial parts construction studied by Etingof and Ginzburg in \cite{EG}.

Inspired by the formalism in \cite[$\mathsection{7}$]{GG}, we define a certain category $\cc$ of holonomic $(\dd(\Gg)\otimes W_{2n})$-modules supported on $X^{nil}$. The simple objects of $\cc$ will be the mirabolic analogs of Lusztig's character sheaves in Type $C$. We construct an exact functor from $\cc$ to $\oo(eH_ce)$, the category $\oo$ of the spherical Cherednik algebra $eH_ce$ defined in \cite{BEG}.

\subsection{Organization}
Here, we give more details about the contents of the paper.

In $\mathsection\ref{sec:subsch},$ we prove Theorem~\ref{theo:subsch}. The proof of the fact that $X^{nil}$ is Lagrangian is by embedding it into the Lagrangian subscheme defined in \cite{GG}. The proof of the rest of the theorem will be seen to be a consequence of this fact and some elementary $\sll_2$-theory.

In $\mathsection\ref{sec:class}$, we prove Lemma~\ref{lem:maintech} and use it to deduce Theorem~\ref{theo:red}. In $\mathsection\ref{sec:RCA}$, we recall some definitions and results about Cherednik algebras and the generalization of Harish-Chandra's radial parts construction in \cite{EG}. In $\mathsection\ref{sec:quant}$, we define the algebras alluded to in the statement of Theorem~\ref{theo:quant}. In $\mathsection\ref{sec:isoone}$ and $\mathsection\ref{sec:isotwo}$, we construct maps between these algebras and prove that they are isomorphisms. Finally, in $\mathsection\ref{sec:QHRF},$ we note some results about the category $\cc$ and provide the construction of the functor from $\cc$ to $\oo(eH_ce)$.

\subsection{Acknowledgments} 
The author would like to thank Victor Ginzburg for motivating this line of work, for helpful suggestions and for his unparalleled patience, especially when this paper was being written. Thanks are also due to Ivan Losev for answering the author's questions about \cite{L}. Finally, discussions with Santosha Pattanayak led to the discovery of an error in an earlier draft of this paper, for which the author is extremely grateful.

\section{The nilpotent subscheme} \label{sec:subsch}
In this section, we consider the reduced subscheme of $X$ defined via:
\[X^{nil}:=\{(x,y,i)\in \Gg\times\Gg\times V: [x,y]+i^2=0 \text{ and } y \text{ is nilpotent}\}.\]
In $\mathsection{\ref{sec:Lag}}$, we'll prove that this is a Lagrangian subscheme of $\Gg\times\Gg\times V$. We will then use this result to provide a new proof of Theorem~\ref{theo:Losmain}. Next, in $\mathsection{\ref{sec:irredcomp}}$, we compute the irreducible components of $X^{nil}$ and show that they are parametrized by partitions of the positive integer $n$.

\subsection{Lagrangian subscheme} \label{sec:Lag}
To state the precise result, we first describe a symplectic structure on $\xx=\Gg\times\Gg\times V$, which we define as $\omega=\omega_1+\omega_2$. Here, $\omega_1$ is the symplectic form on $\Gg\times \Gg$ obtained by identifying it with $T^*(\Gg)$ using the trace form on $\Gg$ and $\omega_2$ is the form on the symplectic vector space $V$. Next, on the scheme $\mm=\gl(V)\times \gl(V)\times V \times V^*$, we have a symplectic form $\omega'$ obtained by identifying it with $T^*(\gl(V)\times V)$ using the trace form on $\gl(V)$.

\begin{theorem} \label{theo:Lag}
The scheme $X^{nil}$ is a Lagrangian complete intersection in $\xx$.
\end{theorem}

Recall that a (possibly singular) subscheme $Z$ of $\xx$ is said to be Lagrangian if at any smooth point $p$ of $Z$, the tangent space $T_pZ$ is a Lagrangian subspace of $T_p\xx$.

\begin{proof}
The scheme $X^{nil}$ is a closed subscheme of $\xx$ defined using $\dim(\Gg)+\frac{1}{2}\dim(V)$ equations. The formula $[x,y]+i^2=0$ gives $\dim(\Gg)$ of these equations, whereas the other $\frac{1}{2}\dim(V)$ equations follow from the nilpotence condition on $y$. So, $\dim(X^{nil})\geq \dim(\xx)-(\dim(\Gg)+\frac{1}{2}\dim(V))=\dim(\Gg)+\frac{1}{2}\dim(V)=\frac{1}{2}\dim(\xx)$. Therefore, to prove the theorem, it suffices to show that $X^{nil}$ is isotropic. For this, consider the embedding:
\[\Phi:\xx=\Gg\times \Gg\times V\lra \gl(V)\times \gl(V)\times V \times V^*=\mm\]
\[(x,y,i)\mapsto (x,y,i_1,i_2),\]
where $i_1=i/2$ and $i_2$ is the symplectic dual of $i$ in $V$. (That is, $i_2$ is the image of $i$ in $V^*$ under the identification $V\simeq V^*$ using the symplectic form.)

Recall from \cite{GG} the scheme of almost commuting matrices $M\sub \mm$:
\[M:=\{(x,y,i,j)\in \gl(V)\times \gl(V)\times V \times V^*: [x,y]+ij=0\}.\]
Also defined in \cite{GG} was the closed subscheme $M^{nil}$ of $M$ obtained by stipulating $y$ to be nilpotent. Then, under the map $\Phi$, we have $\Phi(X)\sub M$ and $\Phi(X^{nil})\sub M^{nil}$.

We claim that with the symplectic forms $\omega$ and $\omega'$ defined above, the map $\Phi$ is a symplectic embedding. To see this, we first observe that we can express the form $\omega'$ as a sum $\omega'_1+\omega'_2$, where $\omega'_1$ is the symplectic form on $\gl(V)\times \gl(V)$ obtained by identifying it with $T^*(\gl(V))$ and $\omega_2'$ is the symplectic form on $V\times V^*$ obtained by identifying it with $T^*(V).$

Then, it is clear that $\omega_1=\Phi^*(\omega_1').$  Next, if $i,j$ are two vectors in $V$, then we have:
\[ \omega_2'((i_1,j_1),(i_2,j_2))=j_1(i_2)-j_2(i_1)=\dfrac{1}{2}\omega_2(i,j)-\dfrac{1}{2}\omega_2(j,i)=\omega_2(i,j),\]
which shows that $\Phi$ preserves the symplectic structure.

By \cite[Theorem 1.1.4]{GG}, we know that that $M^{nil}$ is a Lagrangian subscheme of $\mm$. In particular, it is isotropic. Therefore, by \cite[Theorem 1.3.30]{CG}, we get that $\Phi(X^{nil})$ is an isotropic subscheme of $\mm$, proving that $X^{nil}$ is an isotropic subscheme of $\xx$.
\end{proof}

\begin{corollary} \label{cor:compint}
The scheme $X$ is a complete intersection of dimension $\dim(\Gg)+\dim(V)$.
\end{corollary}

\begin{proof}
The scheme $X^{nil}$ is obtained from $X$ by imposing exactly $n=\frac{1}{2}\dim(V)$ equations, that come from imposing the nilpotence condition on $y$. Hence, as $\dim(X^{nil})=\frac{1}{2}\dim(\xx)=\dim(\Gg)+\frac{1}{2}\dim(V)$, we must have that $\dim(X)\leq \dim(X^{nil})+\frac{1}{2}\dim(V)=\dim(\Gg)+\dim(V)$. But, the scheme $X$ is obtained from $\xx$ by imposing $\dim(\Gg)$ equations, and so, $\dim(X)\geq \dim(\xx)-\dim(\Gg)=\dim(\Gg)+\dim(V).$ Therefore, $X$ is a complete intersection of dimension $\dim(\Gg)+\dim(V).$
\end{proof}

In fact, we can generalize Theorem~\ref{theo:Lag} as follows. Let $\Hh\sub\Gg$ denote a Cartan subalgebra of $\Gg$ and let $W$ be the Weyl group. Consider the composition map $\phi:\Gg\ra\Gg/\!/G\ra\Hh/\!/W$, where the first map is the quotient map and the second one is the Chevalley restriction isomorphism. Then, we can consider the morphism:
\[\pi:X\lra \Hh/\!/W,\]
that sends a triple $(x,y,i)$ to $\phi(y)$. It is clear that
$X^{nil}=\pi^{-1}(\{0\}).$

\begin{proposition} \label{prop:general}
All the fibers of the map $\pi$ are Lagrangian subschemes of $\xx$ and have dimension $\dim(\Gg)+\frac{1}{2}\dim(V)$.
\end{proposition}

\begin{proof}
For any $x\in\Hh/\!/W$, we have the dimension inequality:
\[\dim(\pi^{-1}(\{x\}))\geq \dim(X)-\dim(\Hh/\!/W)=\dim(\Gg)+\dfrac{1}{2}\dim(V)=\dfrac{1}{2}\dim(\xx).\]
Therefore, to show that $\pi^{-1}(\{x\})$ is Lagrangian, it suffices to prove that it is isotropic.

We consider the symplectic embedding $\Phi:\xx\ra\mm$ defined in the proof of Theorem~\ref{theo:Lag}. By Corollary 2.3.4 of \cite{GG}, the image of $\pi^{-1}(\{x\})$ in $\mm$  lies inside a Lagrangian subscheme of $\mm$. As a result, we conclude that the image of $\pi^{-1}(\{x\})$ must be an isotropic subscheme of $\mm$, showing that $\pi^{-1}(\{x\})$ must itself be isotropic.
\end{proof}

\begin{remark}
In fact, by adapting the proof of Proposition 2.3.2 of \cite{GG}, we can prove the following: Corresponding to the Hamiltonian $G$-action on the variety $\xx$, we get a moment map $\mu:\xx\ra \Gg^*\simeq \Gg$ given by the formula $(x,y,i)\mapsto [x,y]+i^2$. Consider the map:
\[\mu\times \pi: \xx=\Gg\times\Gg\times V \lra \Gg\times \Hh/\!/W,\]
that maps a triple $(x,y,i)$ to the pair $([x,y]+i^2, \phi(y))$. Then, this map is a flat morphism. As a corollary of this fact, we also get that the moment map $\mu:\xx\ra \Gg$ if flat.
\end{remark}

With this, we are ready to prove Losev's theorem. Define the scheme:
\[X^{reg}=\{(x,y,i)\in X: y \text{ is regular, semisimple}\}.\]
In other words, $X^{reg}=\pi^{-1}(\Hh^{reg})$, where $\Hh^{reg}$ is the regular semisimple locus of $\Hh$. By Lemma 2.9 of \cite{L}, the scheme $X^{reg}$ is irreducible.

\begin{theorem}
\begin{enumerate}
\item We have $\ol{X^{reg}}=X$. In particular, the scheme $X$ is irreducible.
\item The scheme $X$ is a reduced, complete intersection of dimension $\dim(\Gg)+\dim(V)$.
\end{enumerate}
\end{theorem}

\begin{proof}
By Corollary~\ref{cor:compint}, we already know that $X$ is a complete intersection of $\dim(\Gg)+\dim(V)$.
Consider the big diagonal $\Delta=(\Hh\sm\Hh^{reg})/\!/W$, which is a closed subscheme of $\Hh/\!/W$ of codimension $1$. Then, we have the equality:
\[X=\ol{X^{reg}}\cup \pi^{-1}(\Delta).\]
Since $\dim(\Delta)=\dim(\Hh/\!/W)-1=\frac{1}{2}\dim(V)-1$, by Proposition~\ref{prop:general}, we have $\dim(\pi^{-1}(\Delta))\leq \dim(\Gg)+\dim(V)-1$. However, as $X$ is a complete intersection, any irreducible component must have dimension exactly $\dim(\Gg)+\dim(V)$. Therefore, we conclude that $\ol{X^{reg}}$ must be the only irreducible component, proving that $X=\ol{X^{reg}}$.

Finally, we note that the action of the group $G$ is generically free on $X^{reg}$. It is known that the moment map is a submersion at the points with a free $G$-orbit, and this implies that $X$ is generically reduced.  Therefore, as $X$ is a complete intersection, it must be Cohen-Macaulay, and thus, we conclude that $X$ is reduced (see, for example, \cite[Theorem 2.2.11]{CG}).
\end{proof}

\subsection{Irreducible components of $X^{nil}$} \label{sec:irredcomp}

Let $P_n$ denote the set of all partitions of $2n$ where every odd part occurs an even number of times. Let $\pp_n\sub P_n$ be the subset of those partitions where each part is even. For any $\lambda\in P_n$, let $\nn_{\lambda}$ denote the corresponding nilpotent conjugacy class of $\Gg$ and define:
\[X_{\lambda}:=\{(x,y,i)\in X^{nil}: y\in \nn_{\lambda}\}.\]
Then, it is clear that we have the following disjoint union:
\[X^{nil}=\coprod_{\lambda\in P_n} X_{\lambda}.\]

\begin{theorem}
\begin{enumerate}
\item For each $\lambda\in \pp_n,$ we have $\dim(X_\lambda)=\dim(\Gg)+\frac{1}{2}\dim(V)$.
\item For each $\lambda\in P_n\sm\pp_n$, we have $\dim(X_\lambda)<\dim(\Gg)+\frac{1}{2}\dim(V)$.
\end{enumerate}
\end{theorem}

As $X^{nil}$ has been shown to be a complete intersection of dimension $\dim(\Gg)+\frac{1}{2}\dim(V)$, this theorem implies that the irreducible components of $X^{nil}$ are given exactly by the closures of those $X_{\lambda}$ for which $\lambda\in\pp_n$, completing the proof of Theorem~\ref{theo:subsch}.

\begin{proof}
Let $y$ be any nilpotent element in $\Gg$. Let $\OO$ denote the minimal orbit consisting of rank one elements in $\Gg$. Define the reduced schemes: (Also defined in \cite{L})
\[X_{y}=\{(x,i)\in \Gg\times V:[x,y]+i^2=0\},\]
\[\underline{X_{y}}=\{(x,z)\in \Gg\times \ol{\OO}:[x,y]+z=0\},\]
\[Y_{y}=\ol{\OO}\cap\{[x,y]: x\in \Gg\}.\]
We have maps:
\[\rho_1:X_{y}\lra\underline{X_{y}}\]
\[(x,i)\mapsto (x,i^2),\]
and
\[\rho_2:\underline{X_{y}}\lra Y_{y}\]
\[(x,z) \mapsto z.\]
The map $\rho_1$ is finite, with either one or two points in the fiber at any point depending on whether $i^2$ is zero or not, respectively. The map $\rho_2$ is an affine bundle map which has fibers of dimension equal to $\dim(\mathfrak{z}_{\Gg}(y)).$ Here, $\mathfrak{z}_{\Gg}(y)$ is the centralizer of $y$ in $\Gg$.  Thus, we get that $\dim(X_{y})  = \dim(\mathfrak{z}_{\Gg}(y))+\dim(Y_{y})$. Further, if $y\in\nn_{\lambda}$ for some $\lambda\in P_n$, we have that $\dim(X_{\lambda})=\dim(X_{y})+\dim(\nn_\lambda) = \dim(\Gg)+\dim(Y_y)$. So, in order to prove the theorem, we are required to show that $\dim(Y_y)=\frac{1}{2}\dim(V)$ for $\lambda\in\pp_n$ and $\dim(Y_y)<\frac{1}{2}\dim(V)$ for $\lambda\not\in\pp_n$.

For a fixed nilpotent $y$, by Jacobson-Morozov theorem, we can find an $\sll_2$-triple $(e,f,h)$ in $\Gg$ with $e=y$. Identifying the subspace $\langle e,f,h\rangle\sub\Gg$ with $\sll_2$, we get an $\sll_2$-action on the vector space $V$. By $\sll_2$-theory, the element $h$ acts semisimply on $V$ with integer eigenvalues. Let $V=V_-\oplus V_0\oplus V_+$ be a decomposition of $V$, such that $V_-, V_0$ and  $V_+$ denote the spans of negative, zero and positive eigenspaces of $h$ respectively.

Now, we consider the identification of $\sll_2$-representations $\Gg=\spp(V)=\Sym^2(V)$. By Lemma~\ref{lem:sl2} proved below, for any $v\in V$, we have:
\[v\in V_+ \iff v^2=\ad_e(x)=\ad_y(x)=[y,x] \text{ for some } x\in\Gg.\]
Hence, we conclude that $\dim(Y_y)=\dim(V_+)$. As there is a one-to-one correspondence between positive and negative eigenvectors of $h$, we have that $\dim(V_-)=\dim(V_+)$, and so, $\dim(Y_y)\leq \frac{1}{2}\dim(V)$. The equality holds exactly when $V_0=0$, that is, when $h$ has no zero eigenvalues. Zero eigenvalues for $h$ occur only in irreducible representations having odd dimension. Hence, the dimension inequality becomes an equality exactly when the space $V$ decomposes into a sum of irreducibles each having even dimension. However, irreducible components of $V$ correspond exactly to the Jordan blocks of $e (=y)$. Therefore, the dimension equality holds exactly when each Jordan block of $y$ has even size, that is $y\in\nn_{\lambda}$ for some $\lambda\in\pp_n$, thus completing the proof.
\end{proof}

\begin{lemma} \label{lem:sl2}
Let $V$ be a representation of $\sll_2(\CC)=\langle e,f,h\rangle$. Let $W=\Sym^2(V)$ with the action of $\sll_2$ induced from the one on $V$. Then, for any $v\in V$, we have that $v^2=e\cdot w$ for some $w\in W$ if and only if $v$ lies in the span of the positive eigenspaces of $h$ in $V$.
\end{lemma}

\begin{proof}
Suppose $v$ lies in the span of the positive eigenspaces of $h$ in $V$. Then, we must have that $v^2$ lies in the span of the positive eigenspaces of $h$ in $W$, and thus, by $\sll_2$-theory, we must have $v^2=e\cdot w$ for some $w\in W$.

Now, we prove the converse. Henceforth, instead of working with $\Sym^2(V)$, we will work with $W=V\otimes V$, for notational convenience. As $V\otimes V$ contains $\Sym^2(V)$ as a proper subrepresentation, if $v^2=e\cdot w$ for some $w\in\Sym^2(V)$, we must have that $v\otimes v=e\cdot \tilde w$ for some $\tilde w\in V\otimes V$.

Let $V=\oplus_i V_i$ be {a} decomposition of $V$ into a direct sum of irreducible $\sll_2$-representations. Then, we get a corresponding decomposition $v=\sum_i v_i$ such that $v_i\in V_i$ for all $i$. Then, as $v\otimes v=e\cdot w$, it is clear that there exist $w_i\in V_i\otimes V_i$ such that $v_i\otimes v_i=e\cdot w_i$ for each $i$. Therefore, without loss of generality, we can assume that $V$ is irreducible.

Suppose $\dim(V)=n+1$ for some $n\in \ZZ_{\geq 0}$. Then, there exists an $h$-eigenbasis $\{x_0,x_1,\dots,x_n\}$ of $V$ such that the actions of $e$ and $f$ are given by:
\[e\cdot x_i=x_{i+1}, f\cdot x_j=(jn-j^2+j)x_{j-1},\]
for all $0\leq i\leq n-1$ and $1\leq j\leq n$. (See, for example, \cite[Section 7.2]{Hum})

Write $v=\sum_i c_ix_i$ for $c_i\in \CC$. Then, we have $v\otimes v=\sum_{i,j}c_ic_jx_i\otimes x_j$.  Pick the smallest $k$ such that $c_k\neq 0$. Then, $c_k^2x_k\otimes x_k$ is the summand in $v\otimes v$ having the strictly smallest eigenvalue for the $h$-action. Therefore, as we have that $v\otimes v=e\cdot w$ for some $w\in V\otimes V$, there must exist $w'\in V\otimes V$ such that $x_k\otimes x_k = e\cdot w'$. Hence, without loss of generality, we can assume that $v=x_k$ for some $k$.

So, we are reduced to showing that given $x_k\otimes x_k=e\cdot w$ for some $w\in V\otimes V$, we must have that $h$ acts on $x_k$ with a positive eigenvalue. For the sake of a contradiction, suppose $h$ acts on $x_k$ with a non-positive eigenvalue. Therefore, we must have that $h$ acts on $x_k\otimes x_k$ with a non-positive eigenvalue. As $x_k\otimes x_k=e\cdot w$, we must have that $h$ acts on $w$ with a strictly negative eigenvalue. Let $V'\sub V$ be the subspace spanned by $\{x_0,x_1,\dots,x_{n-1}\}$ (that is, all but the highest weight vector). Then, we must have that $w\in V'\otimes V'\sub V\otimes V$.

Consider the linear function defined via:
\[\lambda:V\otimes V\lra \CC\]
\[x_i\otimes x_j\mapsto (-1)^i.\]
We claim that $\lambda(e\cdot x)=0$ for all $x\in V'\otimes V'$. To see this, we note that $V'\otimes V'$ is spanned by vectors for the form $x_i\otimes x_j$ for $0\leq i,j\leq n-1$. For such $i$ and $j$, we have:
\[e\cdot(x_i\otimes x_j)=x_{i+1}\otimes x_j+x_i\otimes x_{j+i},\]
which makes it clear that $\lambda(e\cdot(x_i\otimes x_j))=0$.

In particular, since $w\in V'\otimes V'$, we must have $\lambda(e\cdot w)=0$. This implies that $\lambda(x_k\otimes x_k)=0$. However, it follows from the definition of $\lambda$ that $\lambda(x_k\otimes x_k)=(-1)^k\neq 0$. This gives a contradiction, and so, the eigenvalue corresponding to $x_k$ must have been positive, completing the proof.
\end{proof}

\section{Hamiltonian reduction}

\subsection{Classical setting} \label{sec:class}
\subsubsection{}In this section, we prove the isomorphism $X/\!/\{\pm 1\}\simeq A$ (see Theorem~\ref{theo:mainclass}). We will first define some notation to formulate the precise statement. All tensor products will be over $\CC$, unless specified otherwise.

Let $\omega$ denote the symplectic form on the vector space $V$. Owing to the natural action of $G$ on the vector space $V$ that preserves $\omega$, we get a moment map:
\[\mu_1:V\lra \Gg^*,\]
that maps the element $v\in V$ to $v^2\in \Gg\simeq \Gg^*$. Here, we use the identification $\Sym^2(V)=\Gg$. We can dualize this map to get a co-moment map:
\[\theta_1:\Gg\lra \CC[V].\]
By \cite[Proposition 1.4.6]{CG}, we have the following formula for this co-moment map: Any element $x\in \Gg$ maps under $\theta_1$ to the polynomial function on $V$ given by:
\[v\mapsto \dfrac{1}{2}\omega(x\cdot v, v)\]
for all $v\in V$. In particular, the image of $\theta_1$ in $\CC[V]$ is exactly the vector space of polynomial functions on $V$ having degree $2$. We can extend the above map multiplicatively to get a map $\CC[\Gg]\simeq\Sym(\Gg)\ra \CC[V]$, which we also call $\theta_1$, whose image is exactly the subalgebra $\CC[V]_{even}$ of polynomials that have even total degree. Define $K:=\ker(\theta_1)\sub \CC[\Gg]$. We give a more explicit description of this ideal $K$ below.

\begin{lemma} \label{lem:rank1}
Let $\ol{\OO}$ denote the closure of the orbit $\OO$ of rank one matrices in $\Gg$, such that the scheme structure on $\ol{\OO}$ is given by the reduced structure on it. Then, the radical ideal in $\CC[\Gg]$ that defines the scheme $\ol{\OO}$ is generated by the $2\times 2$ minors.
\end{lemma}

\begin{proof}
Consider the map:
\[\mu_1: V \lra \Gg^*\simeq \Gg\]
\[ v\mapsto v^2.\]
The image of $\mu_1$ is exactly $\ol{\OO}$. Also, the pre-image of any point in $\OO$ consists of exactly two vectors in $V$ that are negatives of each other, whereas the pre-image of zero is the zero vector. Therefore, we get an induced map from the categorical quotient:
\[\ol{\mu_1}:V/\!/\{\pm 1\} \lra \ol{\OO}.\]
By the above discussion, $\ol{\mu_1}$ is a closed embedding that is a bijection on $\CC$-points. Therefore, as the scheme $\ol{\OO}$ is reduced, the map $\ol{\mu_1}$ is an isomorphism.

Hence, the coordinate ring $\CC[\ol{\OO}]$ of $\ol{\OO}$ is isomorphic to the invariant ring $\CC[V]^{\{\pm1\}}$. Choosing coordinates $p_1, p_2, \dots, p_{2n}$ in $V$, this invariant ring is equal to $\CC[p_1,p_2,\dots,p_{2n}]^{\{\pm 1\}}$. It is clear that this ring is generated by the polynomials $q_{ij}=p_ip_j$ for $1\leq i,j \leq 2n$. Also, by the second fundamental theorem (\cite[Theorem 2.17A]{W}) for the group $O(1)=\{\pm1\}$, the relations between these generators are exactly given by $R_{ijkl}=q_{ij}q_{kl}-q_{il}q_{kj}$ for $1\leq i,j,k,l \leq 2n$. Since the pullbacks of these relations $R_{ijkl}$'s to the coordinate ring of $\Gg$ are exactly given by the $2\times 2$ minors, this shows that the defining ideal is exactly generated by these elements.
\end{proof}

\begin{corollary} \label{cor:classinv}
The map $\theta_1$ induces an isomorphism of algebras:
\[\theta_1:\CC[\Gg]/K\lra \CC[V]^{\{\pm 1\}}=\CC[V]_{even}.\]
The ideal $K$ is the defining ideal of $\ol{\OO}$ in $\CC[\Gg]$ and is generated by the $2\times 2$ minors.
\end{corollary}

\subsubsection{} 
The adjoint action of the group $G=Sp(V)$ on the Lie algebra $\Gg$ gives rise to a Hamiltonian $G$-action on the symplectic variety $T^*\Gg$. Corresponding to this action, we get a moment map:
\[\mu_0: T^*\Gg\lra \Gg^*.\]
If we identify the space $\Gg^*$ with $\Gg$ and $T^*\Gg$ with $\Gg\times \Gg$, the map $\mu_0$ is given explicitly by the commutator map on $\Gg$. We can dualize $\mu_0$ to get a co-moment map $\theta_0:\Sym(\Gg) \ra \CC[T^*\Gg]$. 

Also, we have a diagonal $G$-action on the space $T^*\Gg\times V$. The moment map $\mu_2:T^*\Gg\times V\ra \Gg^*$ for this action is equal to $\mu_0+\mu_1$, whereas the co-moment map $\theta_2:\Sym(\Gg)\ra \CC[T^*\Gg\times V]\simeq \CC[T^*\Gg]\otimes \CC[V]$ is defined via $\theta_2(x)=\theta_0(x)\otimes 1+1\otimes\theta_1(x)$ for all $x\in\Gg$. Explicitly, the map $\mu_2$ is defined via the formula $(x,y,i)\mapsto [x,y]+i^2$.

Now, we can define the schemes that we'll be dealing with:
\begin{itemize}
\item We define the scheme $X \sub T^*\Gg\times V$ as the zero fiber of the moment map $\mu_2$. More precisely, the defining ideal $I$ of $X$ in $\CC[T^*\Gg\times V]$ is the one generated by $\theta_2(\Gg)$ in $\CC[T^*\Gg\times V]$. Using the formula for $\theta_2$ above, this ideal is equal to the one generated by the matrix entries of the expression $[x,y]+i^2$.

\item We define the scheme $A \sub T^*\Gg$ as the pre-image of $\ol{\OO}$ under the moment map $\mu_0$. That is, the defining ideal $J$ of $A$ in $\CC[T^*\Gg]$ is the one generated by $\theta_0(K)$. By Corollary~\ref{cor:classinv}, this ideal is generated by all $2\times 2$ minors of the commutator $[x,y]$ for $(x,y)\in\Gg\times \Gg$.
\end{itemize}

There is a projection morphism $\phi:X \ra A$ that maps a triple $(x,y,i)\in X$ to the pair $(x,y)\in A$. We will show that:

\begin{theorem} \label{theo:mainclass}
The induced morphism on the categorical quotients:
\[\phi:X/\!/\{\pm 1\}\lra A/\!/\{\pm 1\}=A,\]
is an isomorphism.
\end{theorem}

The proof of this theorem will be a consequence of the following linear algebraic lemma, which will also be useful for proving the quantum Hamiltonian reduction statement in $\mathsection{\ref{sec:isotwo}}$. The algebras in the following lemma aren't necessarily assumed to be commutative.

\begin{lemma} \label{lem:maintech}
Let $\aaa$ be an associative $\CC$-algebra generated by a vector space $\mathcal{V}\sub \aaa$. Suppose the linear map $S:\mathcal{V}\ra\mathcal{V}$ that sends $v\mapsto -v$ for $v\in\mathcal{V}$ extends to an algebra anti-automorphism $S:\aaa\ra \aaa$. Let $\bb_0$ and $\bb_1$ be $\CC$-algebras and fix algebra homomorphisms $f_i: \aaa \ra \bb_i$ for $i=0,1$. Suppose $f_1$ is surjective and let $\mathcal{I}=\ker(f_1)$. Let $f_2:\mathcal{V}\ra \bb_0\otimes \bb_1$ be the linear map defined via the formula $v\mapsto f_0(v)\otimes 1+1\otimes f_1(v)$. Then, there exists a vector space isomorphism:
\[\bb_0/(\bb_0\cdot f_0(S(\mathcal{I})))\simeq (\bb_0\otimes \bb_1)/((\bb_0\otimes \bb_1)\cdot f_2(\mathcal{V})),\]
induced by the linear map $\bb_0\ra \bb_0\otimes \bb_1$ that sends $b\mapsto b\otimes 1$ for $b\in \bb_0$.
\end{lemma}

\begin{proof}
By definition, we have that $f_2(v)=f_0(v)\otimes 1+1\otimes f_1(v) \in ((\bb_0\otimes \bb_1)\cdot f_2(\mathcal{V}))$. As $\aaa$ is generated by $\mathcal{V}$, we get a filtration on $\aaa$ such that the $i^{th}$ filtered piece consists of linear combinations of elements that are products of at most $i$ elements from $\mathcal{A}$. Then, by induction on this filtration, 
for any $a\in \aaa$, we can use the above fact to show that:
\[f_0(a)\otimes 1 = 1\otimes f_1(S(a)) \mod ((\bb_0\otimes \bb_1)\cdot f_2(\mathcal{V})).\]

Next, we observe that:
\[\bb_0/(\bb_0\cdot f_0(S(\mathcal{I}))) \simeq \bb_0 \otimes_{\aaa} (\aaa/S(\mathcal{I})).\]
Further, we have that the maps $S$ and $f_1$ induces an algebra anti-isomorphism:
\[\sigma:
\begin{tikzcd}
    \aaa/S(\mathcal{I}) & \aaa/S(S(\mathcal{I}))
	\arrow["\sim"', "S",, from=1-1, to=1-2]
\end{tikzcd}
= \begin{tikzcd}
    \aaa/\mathcal{I} & \bb_1
	\arrow["\sim"', "f_1",, from=1-1, to=1-2]
\end{tikzcd}.\]
We also have an inverse to this map which we denote by $\sigma':\bb_1\ra \aaa/S(\mathcal{I})$.
Now we can define the required isomorphism:
\[\Phi:\bb_0\otimes_{\aaa} (\aaa/S(\mathcal{I})) \lra (\bb_0\otimes \bb_1)/((\bb_0\otimes \bb_1)\cdot f_2(\mathcal{V})),\]
via the formula $b_0\otimes a \mapsto b_0\otimes \sigma(a)$. To see that this is well-defined, we need to verify that $f_0(a)\otimes 1 - 1\otimes a$ maps to $0$ under $\Phi$ for any $a\in \aaa$. This follows from:
\[\Phi(f_0(a)\otimes 1 - 1\otimes a)=f_0(a)\otimes 1 - 1\otimes f_1(S(a))\in ((\bb_0\otimes \bb_1)\cdot f_2(\mathcal{V})).\]
To see that $\Phi$ is an isomorphism, we write down its inverse:
\[\Psi: (\bb_0\otimes \bb_1)/((\bb_0\otimes \bb_1)\cdot f_2(\mathcal{V})) \lra \bb_0\otimes_{\aaa} (\aaa/S(\mathcal{I})),\]
defined by the formula $b_0\otimes b_1 \mapsto b_0\otimes \sigma'(b_1)$. This map is well-defined because under $\Psi$, the space $(\bb_0\otimes \bb_1)\cdot f_2(\mathcal{V})$ maps exactly to the tensor product relations in $\bb_0\otimes_{\aaa} (\aaa/S(\mathcal{I}))$. That $\Phi$ and $\Psi$ are inverses to each other is straighforward.

\end{proof}

\begin{remark}
One can weaken the hypothesis of the lemma and assume that $\bb_0$ and $\bb_1$ are $\aaa$-modules (rather than algebras) and prove a slightly modified statement, but we don't need that generality here.
\end{remark}

\begin{proof}[Proof of Theorem~\ref{theo:mainclass}]
To prove the isomorphsim of schemes, we need to prove that their respective coordinate rings are isomorphic. That is, we need to show that there is an isomorphism:
\[\phi^*:\CC[T^*\Gg]/J = \CC[A]\lra \CC[X]^{\{\pm 1\}} = \Big(\CC[T^*\Gg]\otimes \CC[V])/I\Big)^{\{\pm 1\}},\]
induced by the map $p\mapsto p\otimes 1$ for $p\in \CC[T^*\Gg]$. As the group $\{\pm 1\}$ acts trivially on $\CC[T^*\Gg],$ we have the equality of algebras:
\[\Big(\CC[T^*\Gg]\otimes \CC[V])/I\Big)^{\{\pm 1\}}= (\CC[T^*\Gg]\otimes \CC[V]_{even})/(I^{\{\pm 1\}}).\]
We apply Lemma~\ref{lem:maintech} by taking $\aaa=\Sym(\Gg)$, $\mathcal{V}=\Gg$, $\bb_0=\CC[T^*\Gg]$, $\bb_1=\CC[V]_{even}$ and $f_i=\theta_i$ for $i=0,1,2$. The ideal $\mathcal{I}=\ker(f_1)=\ker(\theta_1)$, in this case, is equal to $K$, which is $S$-invariant as it is generated by homogeneous polynomials of degree 2 by Corollary~\ref{cor:classinv}. Then, the conclusion of Lemma~\ref{lem:maintech} gives us the required isomorphism.
\end{proof}

\begin{remark}
Theorem~\ref{theo:mainclass} is a special case of the following `shifting trick', which is also an immediate consequence of Lemma~\ref{lem:maintech}:

Let $\Gg$ be a reductive Lie algebra. Let $G$ be the corresponding adjoint group and fix $\OO$ to be a $G$-orbit in $\Gg^*$ under the co-adjoint action. Let $Y$ be an affine variety with a Hamiltonian $G$-action. Then, the Hamiltonian reduction of $Y$ at the negative orbit closure $-\ol{\OO}$ is isomorphic to the Hamiltonian reduction of the variety $Y\times \ol{\OO}$ at $0\in\Gg^*$.
\end{remark}

\subsubsection{}
As a consequence of Theorem~\ref{theo:mainclass}, we have an isomorphism of schemes $X/\!/G\simeq A/\!/G$. Therefore, we have morphisms:
\[C/\!/G \stackrel{\psi}{\lra} A/\!/G \stackrel{\phi}{\longleftarrow} X/\!/G.\]
Set-theoretically, these morphisms are as follows: The map $\psi$ sends a pair $(x,y)$ of commuting matrices to the almost commuting pair $(x,y)$. The map $\phi$ sends a triple $(x,y,i)$ to the pair $(x,y)$. We have shown that $\phi$ is an isomorphism. The morphism $\psi$ is an isomorphism too, because it is proven in \cite{L} that we have an isomorphism $C/\!/G \ra X/\!/G$, and that morphism composed with $\phi$ gives $\psi$.

\begin{remark}
The fact that $\psi$ is an isomorphism can also be proven independently by mimicking the proof of Theorem 12.1 of \cite{EG}, making use of Weyl's fundamental theorem of invariant theory for $\Gg=\spp(V)$.
\end{remark}

Combining Theorem~\ref{theo:red} with Theorem 1.3 of \cite{L}, we have the following corollary:

\begin{corollary} \label{lem:combi}
We have an algebra isomorphism:
\[\CC[\Hh\times\Hh]^W=\CC[(\Hh\times\Hh)/\!/W] \stackrel{\sim}{\lra} \CC[A/\!/G]=\Big(\CC[\Gg\times\Gg]/J\Big)^{\Gg}.\]
\end{corollary}

\subsection{Reminder on rational Cherednik algebras and the Harish-Chandra homomorphism} \label{sec:RCA}

\subsubsection{}
Fix a Cartan sublagebra $\Hh\sub \Gg=\spp(V)$ and a root system $R\sub \Hh^*$. The space $\Hh$ (and hence $\Hh^*$) has an action of the Weyl group $W=(\ZZ/(2))^n \rtimes S_n$. For each $\alpha\in R$, let $s_{\alpha}\in W$ denote the reflection of $\Hh$ relative to the root $\alpha$. Fix a $W$-invariant function $c:R\ra \CC$. For root systems of Type $C$, it turns out that there are exactly two $W$-orbits in $R$ given by the set of all long roots and the set of all short roots. Hence, such a $W$-invariant function $c$ can be viewed as a pair of complex numbers $c=(c_{long},c_{short})\in \CC^2$.

We recall from \cite{EG} the definition of the rational Cherednik algebra $H_c$ of Type $C$, which is the one generated by the algebras $\Sym(\Hh), \CC[\Hh] (\simeq \Sym(\Hh^*))$ and the group algebra $\CC[W]$ with defining relations given by:
\[wxw^{-1}=w(x), wyw^{-1}=w(y),\]
\[[y,x]=\langle x,y\rangle - \dfrac{1}{2}\sum_{\alpha\in R}c(\alpha)\langle\alpha,y\rangle\langle x,\alpha^{\vee}\rangle \cdot s_{\alpha},\]
for all $w\in W$, $x\in \Hh^*$ and $y\in \Hh$. Let $e=\frac{1}{|W|}\sum_{w\in W} w$ be the averaging idempotent in $\CC[W]$. Then, we can construct the spherical subalgebra $eH_ce \sub H_c$, which will be our main object of interest. It is known that the subalgebras $\CC[\Hh]^W$ and $\Sym(\Hh)^W$ generate the algebra $eH_ce$ (see \cite[Proposition 4.9]{EG}).

Another description of the algebra $eH_ce$ given in \cite{EG} will be useful for us, and we recall that here. Define the rational Calogero-Moser operator $L_c$, which is a differential operator on the space $\Hh^{reg}$, as follows (see \cite{OP}):
\[L_c:=\Delta_{\Hh}-\dfrac{1}{2}\sum_{\alpha\in R}\dfrac{c(\alpha)(c(\alpha)+1)}{\alpha^2}\cdot(\alpha,\alpha),\]
where $\Delta_{\Hh}$ is the Laplacian operator on the Cartan subalgebra $\Hh$. Let $\dd(\Hh^{reg})_{-}$ denote the subalgebra of $\dd(\Hh^{reg})$ spanned by differential operators $D\in \dd(\Hh^{reg})$ such that $degree(D)+order(D)\leq 0$. Here, $degree$ denotes the degree of the filtration on $\dd(\Hh^{reg})$ with respect to the $\CC^*$ action and $order$ denotes the order of the differential operator. Let $\mathcal{C}_c$ denote the centralizer of the operator $L_c$ in $\dd(\Hh^{reg})_-^W$. Finally, consider the subalgebra $\mathcal{B}_c$ of $\dd(\Hh^{reg})$ generated by $\mathcal{C}_c$ and $\CC[\Hh]^W$, the algebra of $W$-invariant polynomial functions on $\Hh$.

\begin{theorem} \cite[Proposition 4.9]{EG}\label{theo:genRCA}
We have an embedding (known as the Dunkl embedding) $\Theta:eH_ce \hra \dd(\Hh^{reg})^W$ such that $\Theta(\CC[\Hh]^W)=\CC[\Hh]^W$ and $\Theta(\Sym(\Hh)^W)=\cc_c$. Moreover, $\Theta(\Delta_{\Hh})=L_c$. Furthermore, $\Theta$ induces an isomorphism of algebras $eH_ce\simeq \bb_c$.
\end{theorem}
The algebra $H_c$ has a filtration such that all the elements of $W$ and the generators of $\Sym(\Hh^*)$ have degree zero, and the generators of $\Sym(\Hh)$ have degree one. This induces a filtration on the spherical subalgebra $eH_ce$. The algebra $\dd(\Hh^{reg})$ has a filtration given by the order of differential operators. Then, we have the following PBW theorem for Cherednik algebras:

\begin{proposition} \cite[Corollary 4.4]{EG} \label{prop:PBW}
With the above filtrations, we have an isomorphism:
\[\gr(eH_ce) \simeq \CC[\Hh\times\Hh]^W.\]
Furthermore, under this isomorphism, the associated graded version of the Dunkl embedding $\gr(\Theta): \CC[\Hh\times\Hh]^W \simeq\gr(eH_ce)\ra\gr(\dd(\Hh^{reg}))=\CC[\Hh^{reg}\times\Hh^*]^W$ is exactly the algebra homomorphism induced by the natural embedding $\Hh^{reg}\hra \Hh$.
\end{proposition}

\subsubsection{} We recall the construction of the `universal' Harish-Chandra homomorphism from \cite{EG}. Let $\Gg^{rs}$ denote the subset of regular semisimple elements of $\Gg$. Then, we have that $\Hh^{reg}=\Hh\cap \Gg^{rs}$. Next, inside the universal enveloping algebra $\uu\Gg$, consider the space $(\uu\Gg)^{\ad(\Hh)}$ of $\ad(\Hh)$-invariants. Then, $(\uu\Gg)^{\ad(\Hh)}\cdot \Hh$ is a two-sided ideal of the algebra $(\uu\Gg)^{\ad(\Hh)}$, and so, we can define the quotient algebra $(\uu\Gg)_{\Hh}:=(\uu\Gg)^{\ad(\Hh)}/((\uu\Gg)^{\ad(\Hh)}\cdot\Hh).$

By Proposition 6.1 of \cite{EG}, and the paragraph following its proof, there exists a canonical algebra isomorphism:
\[\Psi: \dd(\Gg^{reg})^{\Gg} \lra (\dd(\Hh^{reg})\otimes (\uu\Gg)_{\Hh})^W.\]
Next, fix a $\uu\Gg$-module $\mathfrak{V}$ and let $\mathfrak{V}\langle0\rangle$ denote its zero weight space. Then, by definition, this space $\mathfrak{V}\langle0\rangle$ is acted upon trivially by $\Hh$, and is thus stable under the action of $(\uu\Gg)^{\ad(\Hh)}.$ Hence, we have an action of $(\uu\Gg)_{\Hh}$ on $\mathfrak{V}\langle0\rangle$, giving an algebra homomorphism $\chi: (\uu\Gg)_{\Hh} \ra \End_{\CC}(\mathfrak{V}\langle0\rangle)$. In particular, if $\mathfrak{V}\langle0\rangle$ is one-dimensional, we get a homomorphism $\chi:(\uu\Gg)_{\Hh}\ra\CC$. So, we can compose it with the above algebra isomorphism and restrict the map to $\dd(\Gg)^{\Gg}\sub \dd(\Gg^{reg})^{\Gg}$ to get a map:
\[\Psi_{\mathfrak{V}}:=\chi\circ\Psi:\dd(\Gg)^{\Gg} \lra \dd(\Hh^{reg})^W.\]
This is the Harish-Chandra homomorphism associated with the representation $\mathfrak{V}$. 

Let $\Ann(\mathfrak{V})\sub\uu\Gg$ be the annihilator of $\mathfrak{V}$ and consider the ideal $(\dd(\Gg)\cdot \ad(\Ann(\mathfrak{V})))^{\Gg},$ which is a two-sided ideal inside $\dd(\Gg)^{\Gg}$. Then, it follows from definitions and the proof of \cite[Proposition 6.1]{EG} that the kernel of the homomorphism $\Psi_{\mathfrak{V}}$ contains the ideal $(\dd(\Gg)\cdot \ad(\Ann(\mathfrak{V})))^{\Gg}.$

\subsubsection{}
For explicit computations in subsequent sections, we fix some coordinates. We identify the symplectic vector space $V$ with the space $\CC^{2n}$, which has the standard symplectic form given by the following $2n\times 2n$ skew-symmetric matrix:
\[\begin{pmatrix}0&I\\-I&0\end{pmatrix}.\]
The subspace $L\sub V=\CC^{2n}$, consisting of vectors whose last $n$ coordinates are zero, is a Lagrangian subspace of $V$ under the above symplectic form and the dual vector space $L^*$ can be identified with the space of vectors whose first $n$ coordinates are zero.

The Lie algebra $\spp(V)=\spp_{2n}$ is the space of endomorphisms preserving this form. Then, $\spp(V)$ can be viewed as the subspace of $\gl(V)=\gl_{2n}$ consisting of block matrices of the form:
\[\begin{pmatrix}A&B\\C&D\end{pmatrix},\]
where $A=-D^T, B=B^T$ and $C=C^T$. The subspace $\Hh$ consisting of diagonal matrices is a Cartan subalgera of $\spp_{2n}$. Identifying $\Hh\simeq \Hh^*$ using the trace form, we fix an orthonormal basis $r_1,r_2, \dots,r_n$ of $\Hh^*$ given by $r_i=(E_{i,i}-E_{i+n,i+n})/\sqrt{2}$, where $E_{i,j}\in \gl_{2n}$ is the elementary matrix whose only non-zero entry is in the $i^{th}$ row of the $j^{th}$ column and is equal to 1.

Then, we have a choice of a root system $R$ given by:
\[R:=\Bigg\{\pm\dfrac{(r_i+r_j)}{\sqrt{2}}, \pm\dfrac{(r_i-r_j)}{\sqrt{2}}, \pm\sqrt{2}r_i: 1\leq i< j \leq n\Bigg\}.\]
Here, the long roots are given the vectors $\pm\sqrt{2} r_i$, whereas the rest are all short roots. The set of positive roots $R^+\sub R$ is obtained by replacing all `$\pm$'-signs by `$+$' in the above defintion. Let $\{e_{\alpha}\}_{\alpha\in R}$ denote the set of root vectors in $\Hh$ that form a Cartan-Weyl basis of the Lie algebra $\Gg$ chosen so that $(e_{\alpha},e_{-\alpha})=1$, where $(.,.)$ is the trace form. Then, we can express the $e_{\alpha}$'s explicitly in terms of elementary matrices $E_{i,j}$ as follows:
\[\alpha=\dfrac{(r_i+r_j)}{\sqrt{2}} \implies e_{\alpha}=e_{-\alpha}^T= \dfrac{E_{i,j+n}+E_{j,i+n}}{\sqrt{2}}\]
\[\alpha=\dfrac{(r_i-r_j)}{\sqrt{2}} \implies e_{\alpha}=e_{-\alpha}^T= \dfrac{E_{i,j}-E_{j+n,i+n}}{\sqrt{2}}\]
\[\alpha=\sqrt{2}r_i\implies e_{\alpha}=e_{-\alpha}^T=E_{i,i+n}.\]

\subsection{Definitions and statement of the main theorem} \label{sec:quant}
\subsubsection{}
We set up some notation. Let $\omega$ denote the symplectic form on the vector space $V$. Let $L$ be a fixed Lagrangian subspace of $V$. Then, we define the Weyl algebra $W_{2n}$ as the algebra of polynomial differential operators on $L$, also denoted by $\dd(L)$. Explicitly, the Weyl algebra is an associative $\CC$-algebra generated by the variables $x_1,x_2,\dots,x_n,y_1,y_2,\dots,y_n$ that satisfy the following relations:
\[[x_i,x_j]=0, [y_i,y_j]=0, [y_i,x_j]=\delta_{i,j}, 1\leq i,j \leq n.\]
Here, the $x_i$'s correspond to a choice of coordinates on the vector space $L^*$ and and the $y_i$'s represent the partial derivatives $\partial_{x_i}$'s with respect to $x_i$'s. We have a direct sum decomposition $W_{2n}=W_{2n,even}\oplus W_{2n,odd}$ as vector spaces, where $W_{2n,even}$ is the space spanned by all the monomials in the $x_i$'s and $y_i$'s having even total degree and $W_{2n,odd}$ is the space spanned by the monomials having odd total degree. It is clear that $W_{2n,even}$ is a subalgebra of $W_{2n}$ which is generated by all monomials of the form $x_ix_j, y_iy_j$ and $x_iy_j+y_jx_i$ for $1\leq i,j\leq n$.

Recall from $\mathsection{\ref{sec:class}}$ that corresponding to the $G$ action on $V$, we have a co-moment map:
\[\theta_1:\Gg\lra \CC[V],\]
such that the image of $\theta_1$ in $\CC[V]$ is exactly the vector space of polynomial functions on $V$ having degree $2$.

Next, we note that the algebra $W_{2n}$ is a quantization of $\CC[V]$. To make this statement precise, we define the symmetrization map, which is the following map of vector spaces:
\[\Sym:\CC[V] \lra W_{2n}\]
\[\lambda_1\lambda_2\cdots\lambda_k \mapsto \dfrac{1}{k!}\sum_{\sigma\in S_k} \lambda_{\sigma(1)}\lambda_{\sigma(2)}\cdots\lambda_{\sigma(k)},\]
where each $\lambda_i$ is a linear function on $V$. (Here, we have used that there is a canonical isomorphism of vector spaces $V^* \simeq L\oplus L^*$ using the symplectic form). The map Sym is a vector space isomorphism. Note that both of these spaces have Lie algebra structures, where the Lie bracket on $\CC[V]$ is given by the Poisson bracket and the bracket on $W_{2n}$ is the one induced by the commutator of the associative product. The following lemma follows by a straightforward computation:

\begin{lemma}
The restriction of the map $\Sym$ to the subspace $\CC[V]_{\leq 2}\sub \CC[V]$ of polynomials of degree less than or equal to two is a Lie algebra homomorphism.
\end{lemma}

Now, we define the composition:
\[\Theta_1:=\Sym\circ\text{ }\theta_1:\Gg\lra W_{2n}.\]
Then, $\Theta_1$ is a Lie algebra homomorphism, and so it induces an algebra homomorphism $\uu\Gg\ra W_{2n}$, which we also denote by $\Theta_1$. Note that the image of $\Theta_1$ lies in the even subalgebra $W_{2n,even}$. Define $\kk:=\ker(\Theta_1)\sub \uu\Gg$.

Viewing the Lie algebra $\mathfrak{sp}(V)=\spp_{2n}$ as a subspace of $\gl_{2n}$ using the embedding defined in the previous section, the map $\Theta_1$ can be written in terms of coordinates via the formula:
\[\begin{pmatrix}A & B\\ C & -A^T\end{pmatrix} \mapsto \dfrac{1}{2}\Bigg(\sum_{i,j=1}^n (2a_{i,j}x_iy_j+b_{i,j}x_ix_j-c_{i,j}y_iy_j) + \Tr(A)\Bigg),\]
where $A=(a_{i,j}), B=(b_{i,j})$ and $C=(c_{i,j})$.

Next, we define the action of $\Gg$ on the space $\dd(\Gg)$ of polynomial differential operators on $\Gg$. For this, fix $x\in \Gg$ and consider the linear map:
\[\ad_x:\Gg\lra\Gg\]
\[y\mapsto[x,y].\]
Using the identification $\Gg\simeq \Gg^*$, the assignment $x\mapsto \ad_x$ gives a Lie algebra homomorphism $\Gg\ra \End(\Gg^*)$, where $\End(\Gg^*)$ denotes the space of vector space endomorphisms of $\Gg^*$. Any element in $\End(\Gg^*)$ can be uniquely extended to a derivation on $\Sym(\Gg^*)\simeq\CC[\Gg]$ via Leibniz rule. Since $\Der(\CC[\Gg])\sub\dd(\Gg)$, we have constructed a map:
\[\Theta_0:\Gg\lra \dd(\Gg),\]
which is a Lie algebra homomorphism. Hence, this induces an algebra homomorphism $\uu\Gg\ra \dd(\Gg)$, which we also denote by $\Theta_0$. Then, given $x\in\Gg$ and $d\in\dd(\Gg)$, the action of $\Gg$ on $\dd(\Gg)$ is defined via:
\[x\cdot d:=[\Theta_0(x),d].\]

We now define the algebras we will be working with:

\begin{itemize}
\item Consider the algebra $\dd(\Gg)\otimes W_{2n}$. We have the diagonal $\Gg$-action on this algebra: Given $d\otimes w\in \dd(\Gg)\otimes W_{2n}$ and $x\in \Gg$, we define the action via:
\[x\cdot(d\otimes w):=[\Theta_0(x),d]\otimes w + d\otimes[\Theta_1(x),w].\] 
Let $\Theta_2:\Gg\ra \dd(\Gg)\otimes W_{2n}$ be the Lie algebra homomorphism defined via $\Theta_2=\Theta_0\otimes 1+1\otimes\Theta_1$. This can be extended to get an algebra homomorphism $\Theta_2:\uu\Gg\ra\dd(\Gg)\otimes W_{2n}$, which is the co-moment map for the above action. Then, we have the quantum Hamiltonian reduction $\Big((\dd(\Gg)\otimes W_{2n})/(\dd(\Gg)\otimes W_{2n})\cdot\Theta_2(\Gg)\Big)^{\Gg}$ of $\dd(\Gg)\otimes W_{2n}$ at the augmentation ideal $\uu\Gg^+$ of $\uu\Gg$.

\item Define $\Jj\sub \dd(\Gg)$ to be the left ideal generated by the image of $\kk\sub \uu\Gg$ under the co-moment map $\Theta_0$, that is, $\Jj:=\dd(\Gg)\cdot(\Theta_0(\ker(\Theta_1)))$. Then, we have the algebra $\Big(\dd(\Gg)/\Jj\Big)^{\Gg}$ which is the quantum Hamiltonian reduction of the algebra $\dd(\Gg)$ at the ideal $\kk \sub \uu\Gg$.

\item Finally, we consider the spherical subalgebra $eH_ce$ of the rational Cherednik algebra $H_c$ with the parameter $c$ given by $c=(c_{long},c_{short})=(-1/4,-1/2).$
\end{itemize}

Our goal is to prove Theorem~\ref{theo:quant} by constructing algebra isomorphisms:

\[\Psi:\Big(\dd(\Gg)/\Jj\Big)^{\Gg}\lra eH_ce,\]
\[\Phi:\Big(\dd(\Gg)/\Jj\Big)^{\Gg}\lra\Big((\dd(\Gg)\otimes W_{2n})/(\dd(\Gg)\otimes W_{2n})\cdot\Theta_2(\Gg)\Big)^{\Gg}.\]
These isomorphisms are established in Theorems~\ref{theo:psi} and \ref{theo:phi} respectively.

\subsubsection{}
We define some filtrations on the algebras we will be working with. On the algebra $\dd(\Gg)$, we have an increasing filtration given by the order of the differential operators. The associated graded with respect to this filtration is given by $\gr(\dd(\Gg))\simeq \CC[T^*\Gg]\simeq \CC[\Gg\times\Gg],$ identifying $\Gg^*$ with $\Gg$ using the trace pairing. Similarly, we have a filtration on the algebra $\dd(\Hh)$ by the order of differential operators and $\gr(\dd(\Hh))\simeq\CC[\Hh\times \Hh]$.

On the even part of the Weyl algebra $W_{2n,even}$, we define a version of the Bernstein filtration. Note that the algebra $W_{2n,even}$ is generated by elements of the form $x_ix_j, y_iy_j$ and $x_iy_j+y_jx_i$. We define the filtration on $W_{2n,even}$ by placing each of these elements in degree 1. Thus, any element which is a product of $2m$ of the $x_i$'s and $y_i$'s lies in the $m^{th}$ filtered piece. Under this filtration, we have the associated graded $\gr(W_{2n,even})=\CC[V]_{even}$.

Using the two filtrations above, we also get a tensor product filtration on $\dd(\Gg)\otimes W_{2n,even}$ such that the associated graded is $\gr(\dd(\Gg)\otimes W_{2n,even}) = \CC[\Gg\times \Gg]\otimes \CC[V]_{even}$. Next, on the algebra $\uu\Gg$, we have the PBW filtration, and so by PBW theorem, we have the associated graded with respect to this filtration $\gr(\uu\Gg)\simeq\Sym\Gg$.
 
\begin{remark}
Note that the $\Gg$-action on each of these algebras is filtration preserving, and so, as $\Gg$ is reductive, taking $\Gg$-invariants commutes with taking associated graded.
\end{remark}

Finally, the Cherednik algebra $H_c$ has a filtration such that all the elements of $W$ and the generators of $\Sym(\Hh^*)$ have degree zero, and the generators of $\Sym(\Hh)$ have degree one. This also induces a filtration on the spherical subalgebra $eH_ce$. By Proposition~\ref{prop:PBW}, we have the associated graded $\gr(eH_ce)=\CC[\Hh\times\Hh]^W.$
 
Before moving on to the construction of the maps $\Phi$ and $\Psi$, we prove a lemma that relates the ideals in the non-commutative setting with the ideals in the commutative setting:
 
\begin{lemma} \label{lem:crugr}
We have an inclusion of ideals $J\sub\gr(\Jj)$, where $J\sub\CC[\Gg\times\Gg]$ is the defining ideal of pairs of matrices whose commutator has rank less than or equal to 1.
\end{lemma}

Recall from the definitions given in $\mathsection\ref{sec:class}$ that the ideal $J$ is generated by all $2\times 2$ minors of the commutator $[x,y]$ for $(x,y)\in\Gg\times \Gg$.

\begin{proof}
First, we describe the associated graded versions of the maps $\Theta_1$ and $\Theta_0$, in order to compute $\gr(\Jj)$. The map $\Theta_1:\uu\Gg \ra W_{2n,even}\sub W_{2n}$ was defined on $\Gg\sub\uu\Gg$ as the composition $\Sym \circ\text{ }\theta_1$. Then, this is a filtration-preserving map, and we have:
\[\theta_1=\gr(\Theta_1):\Sym\Gg = \gr(\uu\Gg) \lra \gr(W_{2n,even})=\CC[V]_{even}\sub \CC[V].\]
This induces a map $V=\Spec(\CC[V])\ra\Spec(\Sym \Gg))=\Gg^*\simeq \Gg$ that sends a vector $v$ to the rank one endomorphism $v^2\in\Gg$. By Lemma~\ref{lem:rank1}, we know that the ideal $\ker(\theta_1)$ is generated by all $2\times 2$ minors in the entries of $\Gg$.

Next, we consider the map $\Theta_0:\uu\Gg\ra\dd(\Gg)$. In this case, we have the associated graded map:
\[\theta_0=\gr(\Theta_0):\Sym\Gg = \gr(\uu\Gg) \lra \gr(\dd(\Gg)) = \CC[\Gg\times \Gg].\]
This is the co-moment map for the diagonal adjoint action of $G$ on $\Gg\times \Gg$. This morphism $\theta_0$ induces a map $\Gg\times\Gg = \Spec(\CC[\Gg\times\Gg])\ra\Spec(\Sym\Gg)= \Gg^* \simeq \Gg$ sending a pair $(x,y)$ to the commutator $[x,y]$.

Hence, if we consider the ideal generated by the image of $\ker(\theta_1)$ under the map $\theta_0$, then this is exactly generated by the $2\times 2$ minors of the commutator $[x,y]$ for $(x,y)\in \Gg\times\Gg$. But these are exactly the generators of the ideal $J$, and so $J=\CC[\Gg\times\Gg]\cdot\theta_0(\ker(\theta_1))$. Also, by definition, $\Jj= \dd(\Gg)\cdot\Theta_0(\ker(\Theta_1))$.  Hence, to prove that $J\sub\gr(\Jj)$, it suffices to show that $\gr(\ker(\Theta_1))=\ker(\theta_1)$. It is clear that $\gr(\ker(\Theta_1))\sub \ker(\theta_1)$.

To see that the inclusion is an equality, we first describe the images of the maps $\theta_1$ and $\Theta_1$. We have that $\theta_1(\Gg)$ is exactly the space of degree $2$ polynomials in $\CC[V]$, and so, $\im(\theta_1)$ is equal to the subalgebra $\CC[V]_{even}$ of all polynomials having even total degree. Applying the symmetrization map, we see that $\im(\Theta_1)$ is exactly $W_{2n,even}$. Thus, we have a short exact sequence of filtered $\uu\Gg$-modules:
\[\begin{tikzcd}
0\arrow[r] & \ker(\Theta_1)\arrow[r] & \uu\Gg \arrow[r] & W_{2n,even} \arrow[r]& 0.
\end{tikzcd}\]
We claim that the filtrations on $\ker(\Theta_1)$ and $W_{2n,even}$ are exactly the ones induced on them by $\uu\Gg$. For $\ker(\Theta_1)$, this is true by definition. For $W_{2n,even}$, this follows by observing that the function $\Theta_1$ maps $\Gg\sub \uu\Gg$ to exactly the space spanned by the elements $x_ix_j, y_iy_j$ and $x_iy_j+y_jx_i$. Hence, this is a strict exact sequence of filtered modules, and so, by \cite[Lemma 1]{Gun}, we can take the associated graded to get an exact sequence of $\gr(\uu\Gg)=\Sym\Gg$-modules:
\[\begin{tikzcd}
0\arrow[r] & \gr(\ker(\Theta_1))\arrow[r] \arrow[d, hook] & \gr(\uu\Gg) \arrow[r] \arrow[d, equal] & \gr(W_{2n,even}) \arrow[r] \arrow[d, equal] & 0\\
0\arrow[r] & \ker(\theta_1)\arrow[r]
& \Sym\Gg \arrow[r] & \CC[V]_{even} \arrow[r] & 0\\
\end{tikzcd}.\]
This implies that $\gr(\ker(\Theta_1))=\ker(\theta_1)$, completing the proof.
\end{proof}

\begin{remark}
The ideal $\ker(\Theta_1)$ is a primitive ideal of $\uu\Gg$ and the equality $\gr(\ker(\Theta_1))=\ker(\theta_1)$ can also be seen in a more general setting in the works of Joseph, where he constructs minimal realizations of simple Lie algebras in the Weyl algebra (see \cite{Jo1}, \cite{Jo2}). This ideal is often referred to as the Joseph ideal in the literature.
\end{remark}

\subsection{Construction of the isomorphism $\Psi$} \label{sec:isoone}

In this section, we construct the isomorphism:
\[\Psi:\Big(\dd(\Gg)/\Jj\Big)^{\Gg} \lra eH_ce.\]

For this, we recall the map $\Theta_1:\Gg\ra W_{2n}$ defined in the previous section. The Weyl algebra $W_{2n}$ is the space of polynomial differential operators on the vector space $L$, and we have fixed coordinates $x_1, x_2, \dots, x_n$ on $L^*$. Let $\mathfrak{V}$ be the vector space spanned by all expressions of the form $(x_1x_2\cdots x_n)^{-1/2}\cdot P$, where $P$ is a Laurent polynomial in $x_1, x_2, \dots, x_n$, i.e. $P\in \CC[x_1^{\pm 1}, x_2^{\pm1}, \cdots, x_n^{\pm 1}]$. Then, the map $\Theta_1$ gives an action of $\Gg$ on $\mathfrak{V}$, where any element $f\in \Gg$ acts on $\mathfrak{V}$ by $\Theta_1(f)\in W_{2n}$ via formal differentiation of Laurent polynomials.

Recall that the Cartan $\Hh \sub \Gg$ is spanned by elements of the form $E_{i,i} - E_{n+i,n+i}$ for $1\leq i\leq n$. Then, under $\Theta_1$, the image of this element is the differential operator $(x_iy_i+y_ix_i)/2 = x_iy_i + \frac{1}{2} = x_i\partial_{x_i}+ \frac{1}{2}$. Hence, the $\Gg$-action on $\mathfrak{V}$ has a one-dimensional zero weight space $\mathfrak{V}\langle 0\rangle$ spanned by the element $(x_1x_2\cdots x_n)^{-1/2}\cdot 1$. Then, using the construction of the radial parts homomorphism in $\mathsection\ref{sec:RCA}$, we get an algebra homomorphism:
\[\Psi: \dd(\Gg)^{\Gg} \lra \dd(\Hh^{reg})^W.\]

Let $\Delta_{\Gg}$ and $\Delta_{\Hh}$ be the Laplacian operators on $\Gg$ and $\Hh$ respectively. Recall from $\mathsection\ref{sec:RCA}$ the Caloger-Moser differential operator $L_c$ for $c=(c_{long},c_{short})=(-1/4,-1/2)$ and let $\mathcal{C}_c$ denote the centralizer of $L_c$ in $\dd(\Hh^{reg})^W$. Also, let $\Sym(\Gg) \sub \dd(\Gg)$ denote the subalgebra of differential operators on $\Gg$ with constant coefficients.

The computations in the following lemma are the origin of the precise value of our parameter $c$.

\begin{lemma} \label{lem:bestlem}
\begin{enumerate} [label=(\alph*)]
\item We have the equality:
\[\Psi(\Delta_{\Gg})=L_c.\]
\item The map $\Psi$ induces an isomorphism:
\[\Sym(\Gg)^{\Gg} \stackrel{\sim}{\lra} C_c.\]
\end{enumerate}
\end{lemma}

\begin{proof}
\begin{enumerate} [label=(\alph*)]
\item By Proposition 6.2 of \cite{EG}, we have:
\[\Psi(\Delta_{\Gg})=\Delta_{\Hh} - \sum_{\alpha\in R}\dfrac{e_{\alpha}\cdot e_{-\alpha}}{\alpha^2}.\]
We will evaluate $e_{\alpha}\cdot e_{-\alpha}|_{\mathfrak{V}\langle0\rangle}$ for the long roots and the short roots separately:
\begin{enumerate} [label=(\roman*)]
\item Suppose $\alpha$ is a long root. (For such roots, $(\alpha,\alpha)=2$.) Hence, $\alpha=\sqrt2 r_i$ for some $i$, where $r_i=(E_{i,i}-E_{i+n,i+n})/\sqrt{2}$. Then, $e_{\alpha}=e_{-\alpha}^T=E_{i,i+n}$, which implies that:
\[\Theta_1(e_{\alpha}\cdot e_{-\alpha}) = \dfrac{({x_i}^2)\cdot(-\partial_{x_i}^2)}{4}.\]
It is straightforward to see that this operator acts by $\frac{-3}{16}\Id_{\mathfrak{V}\langle0\rangle}=c_{long}(c_{long}+1)\Id_{\mathfrak{V}\langle0\rangle}$ on the space spanned by $(x_1x_2\cdots x_n)^{-1/2}\cdot 1$.
\item Suppose $\alpha$ is a short root. (For such roots, $(\alpha,\alpha)=1$.) Then, $\alpha=(r_i+ r_j)/\sqrt 2$ or $\alpha=(r_i- r_j)/\sqrt 2$ for some $i\neq j$. In the first case, $e_{\alpha}=e_{-\alpha}^T=(E_{i,j+n}+E_{j,i+n})/\sqrt{2}$, and so:
\[\Theta_1(e_{\alpha}\cdot e_{-\alpha}) = \dfrac{({x_i}{x_j})\cdot(-\partial_{x_i}\partial_{x_j})}{2}.\]
In the second case, $e_{\alpha}=e_{-\alpha}^T=(E_{i,j}+E_{j+n,i+n})/\sqrt{2}$, giving that:
\[\Theta_1(e_{\alpha}\cdot e_{-\alpha}) = \dfrac{(x_i\partial_{x_j})\cdot(x_j\partial_{x_i})}{2}.\]
Then, in either case, the differential operator acts by $\frac{-1}{8}\Id_{\mathfrak{V}\langle0\rangle}=\frac{1}{2}c_{short}(c_{short}+1)\Id_{\mathfrak{V}\langle0\rangle}$ on the space spanned by $(x_1x_2\cdots x_n)^{-1/2}\cdot 1$.
\end{enumerate}

Hence, we conclude that:
\[\Psi(\Delta_{\Gg})=\Delta_{\Hh}-\dfrac{1}{2}\sum_{\alpha\in R}\dfrac{c(\alpha)(c(\alpha)+1)}{\alpha^2}\cdot(\alpha,\alpha),\]
which is exactly the Calogero-Moser operator $L_c$ of Type $C$ for the parameter $c$.

\item This follows from the proof of Proposition 7.2 of \cite{EG}, which works for any reductive Lie algebra $\Gg$.
\end{enumerate}
\end{proof}

By the above lemma, we see that $\mathcal{C}_c \sub \im(\Psi)$. Next, the restriction of $\Psi$ to $\CC[\Gg]^{\Gg} \sub \dd(\Gg)^{\Gg}$ is exactly the Chevalley restriction map $\CC[\Gg]^{\Gg}\ra \CC[\Hh]^W\sub \dd(\Hh^{reg})$, and so, $\CC[\Hh]^W \sub \im(\Psi)$. Now, by Theorem~\ref{theo:genRCA}, the image of the spherical subalgebra $eH_ce$ under the Dunkl homomorphism $\Theta:eH_ce\ra \dd(\Hh^{reg})^W$ is exactly the subalgebra of $\dd(\Hh^{reg})^W$ generated by $\mathcal{C}_c$ and $\CC[\Hh]^W$. Hence, we have the following corollary:

\begin{corollary} \label{cor:Dunkl}
We have the inclusion of algebras $\Theta(eH_ce)\sub \im(\Psi)$.
\end{corollary}

Next, let $\Ann(\mathfrak{V})\sub \uu\Gg$ denote the annihilator of the representation $\mathfrak{V}$. Then, as remarked in $\mathsection\ref{sec:RCA}$, we have that $(\dd(\Gg)\cdot \ad(\Ann(\mathfrak{V})))^{\Gg} \sub \ker(\Psi)$. Finally, we note that the action of $\Gg$ on $\mathfrak{V}$ was defined via the map $\Theta_1$, and so, $\ker(\Theta_1)\sub \Ann(\mathfrak{V})$, which implies that $\Jj^{\Gg} = (\dd(\Gg)\cdot \Theta_0(\ker(\Theta_1)))^{\Gg} \sub (\dd(\Gg)\cdot \ad(\Ann(\mathfrak{V})))^{\Gg}$.

\begin{theorem} \label{theo:psi}
We have an isomorphism of algebras:
\[\Big(\dd(\Gg)/\Jj\Big)^{\Gg} \simeq eH_ce.\]
\end{theorem}

\begin{proof}
The proof of this theorem is based on a commutative diagram that is very similar to the one present in the proof of Theorem 1.3.1 of \cite{GG}.

As noted above, we have $\Jj^{\Gg} \sub (\dd(\Gg)\cdot \ad(\Ann(\mathfrak{V})))^{\Gg}\sub\ker(\Psi)$, and so, the map $\Psi$ can be factored to induce an algebra hommomorphism (which we also denote by $\Psi$):
\[\Psi:\Big(\dd(\Gg)/\Jj\Big)^{\Gg} \lra \dd(\Hh^{reg}).\]

Identifying $\Hh$ with $\Hh^*$ using the trace form, we get an algebra isomorphism $\phi:\CC[\Hh\times\Hh^*]\ra \CC[\Hh\times\Hh]$. Then, we have the following diagram:
\[\begin{tikzcd}[ampersand replacement=\&]
	{\mathbb{C}[\mathfrak{h}\times\mathfrak{h}^*]^W} \&\& {\mathbb{C}[\mathfrak{h}\times\mathfrak{h}]^W} \&\& {\Big(\mathbb{C}[\mathfrak{g}\times\mathfrak{g}]/J\Big)^{\mathfrak{g}}} \\
	{\text{gr}(eH_ce)} \&\&\&\& {\Big(\text{gr}(\mathcal{D}(\mathfrak{g}))/\text{gr}(\mathfrak{J})\Big)^{\mathfrak{g}}} \\
	{\text{gr}(\Theta(eH_ce))} \&\& {\text{gr}\Big(\Psi\Big(\dd(\Gg)/\Jj\Big)^{\Gg}\Big)} \&\& {\text{gr}\Big(\dd(\Gg)/\Jj\Big)^{\mathfrak{g}}}
	\arrow["\sim"', "\phi",from=1-1, to=1-3]
	\arrow["\sim"', "\text{Cor.}~\ref{lem:combi}",from=1-3, to=1-5]
	\arrow["\sim", "\text{Proposition}~\ref{prop:PBW}"', from=1-1, to=2-1]
        \arrow["\sim", "\text{Proposition}~\ref{prop:PBW}"', from=2-1, to=3-1]
	\arrow["\text{Lemma}~\ref{lem:crugr}", two heads, from=1-5, to=2-5]
	\arrow["\text{proj}", two heads, from=2-5, to=3-5]
	\arrow[hook, "\text{Cor.}~\ref{cor:Dunkl}", from=3-1, to=3-3]
	\arrow["\text{gr}\Psi"', , from=3-5, to=3-3]
\end{tikzcd}.\]
This diagram commutes, and so, we get, in particular, that the map $\gr(\Psi)$ must be injective. Hence, the map $\Psi$ is an isomorphism onto its image. This along with the injectivity of $\gr(\Psi)$ implies that $\gr(\Psi)$ is in fact bijective.

Hence, all the maps in the above diagram are bijections. In particular, the image of $\gr(\Psi)$ in $\gr(\dd(\Hh^{reg}))=\CC[\Hh\times\Hh^{reg}]$ can be identified with the image of $\gr(\Theta)$. This identification can also be obtained as the associated graded of the embedding $\Theta(eH_ce)\sub \im(\Psi)$ from Corollary~\ref{cor:Dunkl}, and so, this embedding must itself be an equality. Hence, we can compose with $\Theta^{-1}$ (as $\Theta$ is injective) to get an algebra homomorphism:
\[\Theta^{-1}\circ \Psi:\Big(\dd(\Gg)/\Jj\Big)^{\Gg} \lra eH_ce.\]
It follows from the commutative diagram that the associated graded version of this map gives the bijection between $\gr\Big((\dd(\Gg)/\Jj)^{\Gg}\Big)$ and $\gr(eH_ce)$. Hence, the map $\Theta^{-1}\circ\Psi$ itself must be a bijection, which is exactly the claim of the theorem.
\end{proof}

\begin{corollary}
\begin{enumerate}
\item We have an isomorphism of commutative algebras:
\[\gr\Big(\dd(\Gg)/\Jj\Big)^{\Gg}\stackrel{\sim}{\lra}\gr(eH_ce).\]
\item All the maps in the above commutative diagram are isomorphisms. In particular, we get that:
\begin{enumerate}
\item We have an isomorphism:
\[\CC[A/\!/G]=\Big(\CC[\Gg\times \Gg]/J\Big)^{\Gg}\simeq \gr\Big(\dd(\Gg)/\Jj\Big)^{\Gg}.\]
\item We have the equality of ideals $\gr(\Jj)^{\Gg}=J^{\Gg}$ in the ring $\CC[\Gg\times\Gg]^{\Gg}$ (cf. Lemma~\ref{lem:crugr}). 
\end{enumerate}
\end{enumerate}
\end{corollary}

\subsection{Construction of the isomorphism $\Phi$} \label{sec:isotwo}

Now, we construct an isomorphism of non-commutative algebras:

\[\Phi:\Big(\dd(\Gg)/\Jj\Big)^{\Gg}\lra \Big((\dd(\Gg)\otimes W_{2n})/(\dd(\Gg)\otimes W_{2n})\cdot\Theta_2(\Gg)\Big)^ {\Gg}.\]

In fact, we prove the following stronger result:

\begin{theorem} \label{theo:phi}
There is an isomorphism of vector spaces:
\[\Phi:\dd(\Gg)/\Jj \lra \Big((\dd(\Gg)\otimes W_{2n})/(\dd(\Gg)\otimes W_{2n})\cdot\Theta_2(\Gg)\Big)^{\{\pm 1\}}.\]
Furthermore, the map $\Phi$ restricts to an algebra isomorphism between the respective subspaces of $\Gg$-invariants.
\end{theorem}

\begin{proof}
As the group $\{\pm 1\}$ acts trivially on the algebra $\dd(\Gg)$, we have an equality of vector spaces:
\[\Big((\dd(\Gg)\otimes W_{2n})/(\dd(\Gg)\otimes W_{2n})\cdot\Theta_2(\Gg)\Big)^{\{\pm 1\}}=(\dd(\Gg)\otimes W_{2n,even})/(\dd(\Gg)\otimes W_{2n,even})\cdot\Theta_2(\Gg).\]
We apply Lemma~\ref{lem:maintech} taking $\aaa=\uu\Gg$, $\mathcal{V}=\Gg$, $\bb_0=\dd(\Gg)$, $\bb_1=W_{2n,even}$ and $f_i=\Theta_i$ for $i=0,1,2$. Here, the ideal $\mathcal{I}=\ker(f_1)=\ker(\Theta_1)$ is equal to $\kk$, which is $S$-invariant because we have a commutative diagram:
\[\begin{tikzcd}[ampersand replacement=\&]
	{\mathcal{U}\mathfrak{g}} \& {W_{2n}} \\
	{\mathcal{U}\mathfrak{g}} \& {W_{2n}}
	\arrow["{S}"', from=1-1, to=2-1]
	\arrow["{\Theta_1}", from=1-1, to=1-2]
	\arrow["{S'}", from=1-2, to=2-2]
	\arrow["{\Theta_1}", from=2-1, to=2-2]
\end{tikzcd},\]
where the right vertical map $S':W_{2n}\ra W_{2n}$ is an algebra anti-homomorphism defined by sending the generators $x_1,\dots,x_n,y_1,\dots,y_n$ of $W_{2n}$ to $ix_1,\dots ix_n,iy_1,\dots,iy_n$ respectively, where $i=\sqrt{-1}$. Then, the conclusion from Lemma~\ref{lem:maintech} gives us the required vector space isomorphism.

Furthermore, as this isomorphism is induced by the map $\dd(\Gg)\ra \dd(\Gg)\otimes W_{2n}$ given by $D\mapsto D\otimes 1$ for $D\in\dd(\Gg)$, which is an algebra homomorphism, the restriction to $\Gg$-invariants gives an algebra isomorphism.
\end{proof}

\begin{corollary}
We have isomorphisms of commutative algebras:
\[\gr\Big(\dd(\Gg)/\Jj\Big)^{\Gg}\simeq\gr\Big((\dd(\Gg)\otimes W_{2n})/(\dd(\Gg)\otimes W_{2n})\cdot\Theta_2(\Gg)\Big)^{\Gg}\simeq \Big(\CC[\Gg\times\Gg\times V]/I\Big)^{\Gg}=\CC[X/\!/G].\]
\end{corollary}

\begin{proof}
Recall from $\mathsection{\ref{sec:class}}$ that we have an isomorphism of schemes $\phi:X/\!/G \ra A/\!/G$. This gives an isomorphism between the coordinate rings: \[\phi^*:\Big(\CC[\Gg\times\Gg]/J\Big)^{\Gg}=\CC[A/\!/G]\lra \CC[X/\!/G]=\Big(\CC[\Gg\times\Gg\times V]/I\Big)^{\Gg}.\]
Then, we can consider the commutative diagram:
\[\begin{tikzcd}[ampersand replacement=\&] {\Big(\CC[\Gg\times\Gg]/J\Big)^{\Gg}} \& {\Big(\CC[\Gg\times\Gg\times V]/I\Big)^{\Gg}}\\
	{\gr\Big(\dd(\Gg)/\Jj\Big)^{\Gg}} \& {\gr\Big((\dd(\Gg)\otimes W_{2n})/(\dd(\Gg)\otimes W_{2n})\cdot\Theta_2(\Gg)\Big)^{\Gg}}
	\arrow["\phi^*", "\sim"', from=1-1, to=1-2]
	\arrow["proj", two heads, from=1-1, to=2-1]
	\arrow["\gr(\Phi)", from=2-1, to=2-2]
	\arrow["proj", two heads, from=1-2, to=2-2]
\end{tikzcd},\]
In this diagram, the top and the left maps are already known to be bijective. We have now shown that $\Phi$ is bijective too. Also, it's clear by unwrapping the defintions that $\Phi$ is a filtration preserving map, and so is its inverse. Therefore, we conclude that $\gr(\Phi)$ must also be a bijection. This forces the fourth map in the commutative diagram to be a bijection.
\end{proof}
\subsection{Quantum Hamiltonian reduction functor} \label{sec:QHRF}

Let $(\dd(\Gg)\otimes W_{2n})$-$mod$ denote the category of finitely generated $(\dd(\Gg)\otimes W_{2n})$-modules. The algebra $\dd(\Gg)\otimes W_{2n}$ contains the subalgebra $Z= \Sym(\Gg)^{\Gg}$ of invariant differential operators on $\Gg$ with constant coefficients. Let $Z_+\sub Z$ be the augmentation ideal, consisting of differential operators with zero constant term.

Furthermore, we have the algebra homomorphism $\Theta_2:\uu\Gg \ra \dd(\Gg)\otimes W_{2n}$. Also, let $eu\in \dd(\Gg)\sub \dd(\Gg)\otimes W_{2n}$ denote the Euler vector field on $\Gg$. Let $\uu$ be the subalgebra of $\dd(\Gg)\otimes W_{2n}$ generated by the image of $\Theta_2$ and $eu$.

\begin{definition}
Let $\cc$ be the full subcategory of $(\dd(\Gg)\otimes W_{2n})$-$mod$, whose objects are $(\dd(\Gg)\otimes W_{2n})$-modules $M$, such that the action on $M$ of the subalgebra $Z_+$ is locally nilpotent and the action of $\uu$ is locally finite. The category $\cc$ will be referred to as the category of admissible $(\dd(\Gg)\otimes W_{2n})$-modules.
\end{definition}

For any $M\in \cc$, we have a $\Gg$-action on $M$ via the map $\Theta_2$. As the group $G=Sp(V)$ is simply connected, the action of $\uu$ on $M$ being locally finite implies that the $\Gg$-action on $M$ can be integrated to get a rational representation of the group $G$ on $M$. Thus, the local finiteness condition in the above definition implies that the modules $M$ are $G$-equivariant.

We can identify $\dd(\Gg)\otimes W_{2n}$ with the ring of differential operators $\dd(\Gg\times L)$. Taking the order filtration on this algebra, for any finitely generated $\dd(\Gg)\otimes W_{2n}$-module $M$, there exists a characteristic variety $\Ch(M)\sub T^*(\Gg\times L)\simeq \Gg\times\Gg\times V$.

\begin{proposition}
For any $G$-equivariant module $M\in(\dd(\Gg)\otimes W_{2n})$-$mod$, we have $M\in \cc$ if and only if $\Ch(M)\sub X^{nil}.$ Furthermore, all the objects in $\cc$ are holonomic $(\dd(\Gg)\otimes W_{2n})$-modules.
\end{proposition}

\begin{proof}
The proof of the fact that $M \in \cc$ if and only if $\Ch(M)\sub X^{nil}$ follows by essentially repeating the proof of Proposition 5.3.2 of \cite{GG} replacing $\gl(V)\times \PP$ by $\spp(V)\times L$ everywhere. The holonomicity of the objects in $\cc$ follows from the fact that their characteristic variety lies in $X^{nil}$, which is a Lagrangian subvariety of $\Gg\times\Gg\times V$.
\end{proof}

We now define the quantum Hamiltonian reduction functor. Let $Q$ be the quotient $(\dd(\Gg)\otimes W_{2n})/((\dd(\Gg)\otimes W_{2n})\cdot\Theta_2(\Gg))$. We have the quantum Hamiltonian reduction of $\dd(\Gg)\otimes W_{2n}$ with respect to the $\Gg$-action, given by $\aaa:=Q^{\Gg}=\Big((\dd(\Gg)\otimes W_{2n})/(\dd(\Gg)\otimes W_{2n})\cdot\Theta_2(\Gg)\Big)^{\Gg}$. By Theorem~\ref{theo:quant}, we have an isomorphism $\aaa\simeq eH_ce$, where $eH_ce$ is the spherical Cherednik algebra with parameter $c=(-1/4,-1/2)$. Let $eH_ce\text{-}mod$ be the category of finitely generated $eH_ce$-modules.
Then, by Proposition 7.2.2 and Corollary 7.2.4 of \cite{GG}, we have:
\begin{proposition}
\begin{enumerate}
\item The space $Q$ is a finitely generated $(\dd(\Gg)\otimes W_{2n})$-module.
\item There is an exact functor $\HH$:
\[\HH:(\dd(\Gg)\otimes W_{2n})\text{-mod}\lra eH_ce\text{-mod}\]
\[M\mapsto \Hom_{\dd(\Gg)\otimes W_{2n}}(Q,M)=M^{\Gg}.\]
\item The functor $\HH$ has a left adjoint $\prescript{T}{}\HH$:
\[\prescript{T}{}\HH:eH_ce\text{-mod} \lra (\dd(\Gg)\otimes W_{2n})\text{-mod}\]
\[M\mapsto Q\otimes_{\aaa}M,\]
such that the canonical adjunction morphism $M\ra \HH(\prescript{T}{}\HH(M))$ is an isomorphism for all $M\in eH_ce\text{-mod}$.
\item The full subcategory $\ker(\HH)$ is a Serre subcategory of $(\dd(\Gg)\otimes W_{2n})\text{-mod}$ and the functor $\HH$ induces an equivalence of categories:
\[(\dd(\Gg)\otimes W_{2n})\text{-mod}/\ker(\HH)\simeq eH_ce\text{-mod}\]
\end{enumerate}
\end{proposition}

Next, we recall that the algebra $eH_ce$ contains the subalgebra $\Sym(\Hh)^W$. Let $\Sym(\Hh)^W_+$ denote the augmentation ideal of $\Sym(\Hh)^W$. Let $\oo(eH_ce)$ be the category of $\oo$ of the algebra $eH_ce$, which is defined as the full subcategory of $eH_ce\text{-}mod$ whose objects are finitely generated $eH_ce$-modules with locally nilpotent action of $\Sym(\Hh)_+^W\sub eH_ce$ (see, for example, \cite[$\mathsection{2}$]{BEG}).

\begin{proposition} \label{prop:final}
The functor $\HH$ restricts to an exact functor $\HH:\cc\ra \oo(eH_ce)$. This induces an equivelence of categories $\cc/\ker(\HH)\simeq \oo(eH_ce).$
\end{proposition}

\begin{proof}
Under the isomorphism $\aaa\simeq eH_ce$, the subalgebra $Z_+$ of $\aaa$ is mapped exactly to the subalgebra $\Sym(\Hh)_+^W$ of $eH_ce$. Therefore, for any $M\in\cc$, we have $\HH(M)\in\oo(eH_ce)$. Furthermore, by the corollary to Lemma 2.5 of \cite{BEG}, there exists an element $h\in H_c$ that acts locally finitely on every element of $\oo(H_c)\simeq\oo(eH_ce)$, such that the image of the Euler vector field $eu$ in $eH_ce$ is equal to $h$ upto an additive scalar factor (see the proof of Formula 6.7 in \cite{BEG}). Thus, the local finiteness condition on the $eu$-action in the definition of $\cc$ is automatically true in $\oo(eH_ce)$.
\end{proof}

\printbibliography

\end{document}